%% file: Article-PHS_Traffic_REV1.tex
\documentclass[12pt]{iopart}
\usepackage{tikz,bbm}
\expandafter\let\csname equation*\endcsname\relax
\expandafter\let\csname endequation*\endcsname\relax
\usepackage{amsmath,amsthm}

\newcommand{\innerprod}[2]{\langle #1, #2 \rangle}
\newtheorem{theorem}{Theorem}[section]

\newtheorem{lemma}[theorem]{Lemma}
\newtheorem{proposition}{Proposition}

\newtheorem{remark}{Remark}

\usepackage{iopams}

\newcommand{\SpaceTime}[3]{\begin{minipage}[t]{.32\textwidth}
\begin{center}
\small$\alpha=#2$#3\\[.5mm]
\rotatebox{90}{\hspace{1.4cm}\footnotesize Time}\hspace{-1mm}
\rotatebox{90}{$\xrightarrow{\hspace{3.8cm}}$}\hspace{-.5mm}
$\underrightarrow{\includegraphics[width=.82\textwidth]{#1}~~}$\\
\footnotesize Positions
\end{center}
\end{minipage}}

\begin{document}

\title[Stability of a stochastic port-Hamiltonian car-following model]{Stability analysis of a stochastic port-Hamiltonian car-following model}

\author{Barbara R\"udiger}
\address{Group for Stochastic, IMACM, University of Wuppertal, Germany}
\ead{ruediger@uni-wuppertal.de}

\author{Antoine Tordeux}
\address{Group for Traffic Safety and Reliability, IMACM, University of Wuppertal, Germany}
\ead{tordeux@uni-wuppertal.de}

\author{Baris E. Ugurcan}
\address{Group for Stochastic, IMACM, University of Wuppertal, Germany}
\ead{beu4@cornell.edu, ugurcan@uni-wuppertal.de}

\vspace{10pt}
\begin{indented}
\item[]June 2024
\end{indented}

\begin{abstract}
Port-Hamiltonian systems are pertinent representations of many nonlinear physical systems. In this study, we formulate and analyse a general class of stochastic car-following models with a systematic port-Hamiltonian structure. The model class is a generalisation of classical car-following approaches, including the \textit{optimal velocity model} of Bando et al. (1995), the \textit{full velocity difference model} of Jiang et al. (2001), and recent stochastic following models based on the Ornstein-Uhlenbeck process. In contrast to traditional models where the interaction is totally asymmetric (i.e., depending only on the speed and distance to the predecessor), the port-Hamiltonian car-following model also depends on the distance to the follower. We determine the exact stability condition of the finite system with $N$ vehicles and periodic boundaries. The stable system is ergodic with a unique Gaussian invariant measure. Other properties of the model are studied using numerical simulation. It turns out that the Hamiltonian component improves the flow stability and reduces the total energy in the system. Furthermore, it prevents the problematic formation of stop-and-go waves with oscillatory dynamics, even in the presence of stochastic perturbations.\\[2mm]
\end{abstract}


\section{Introduction} 

The stability of single-file motion in one dimension is an active research area in traffic engineering. 
It relies on the collective behaviour of vehicle platoons and the formation of stop-and-go waves in traffic flows. 
Pioneering work by Pipes and later Herman, Gazis, and other authors has shown since the 1950s that car-following models can present stability issues \cite{Pipes1953,Chandler1958,Herman1959,gazis1961nonlinear}. 
Nowadays, it is widely accepted that waves occur spontaneously when the driver's reaction time exceeds a critical threshold due to instability phenomena \cite{bando1995dynamical,Orosz2009,Orosz2010}. 
The waves result from a linear instability inducing a first-order phase transition in a deterministic nonlinear framework. 
Another modelling approach originally based on cellular automata and dating back to the 1990s assumes that the waves result from noise-induced effects \cite{NagelS92A,barlovic1998metastable}. 
Currently, continuous noise models, and in particular overdamped Brownian noise provided by the Ornstein-Uhlenbeck process, can initiate stop-and-go phenomena as second-order phase transitions and subcritical instabilities \cite{treiber2009hamilton,Hamdar2015,tordeux2016white,treiber2017intelligent,wang2020stability,friesen2021spontaneous}. 
Other continuous approaches rely on state-dependent noise using the Cox-Ingersoll-Ross process, which induces instability phenomena \cite{ngoduy2019langevin,xu2019analysis}.

Although traffic flow stability has been studied for up to 70 years, the spontaneous formation and control of stop-and-go waves remains not fully understood, either experimentally or theoretically. 
This topic remains a challenge, especially for driving automation. 
Recent experiments have shown that the platoon dynamics of vehicles equipped with adaptive cruise control (ACC) driving assistance systems currently available on the market exhibit unstable characteristics \cite{Stern2018,gunter2020commercially,makridis2021openacc,CIUFFO2021}.
Considerable efforts are currently being made to ensure that car-following models and ACC systems are not only efficient, but also stable and robust to perturbations. 
Stabilisation properties can be achieved through temporal and spatial anticipation mechanisms, and other compensators \cite{treiber2006delays,wang2019effect,khound2021extending}. 
Autonomous sensing systems rely on next-neighbour interaction, while connected, cooperative systems allow for larger interaction ranges. 

Port-Hamiltonian systems (pHS) have recently been introduced for the modelling of nonlinear physical systems \cite{van2006port,van2014port}. 
The framework dates back to the 1980s and the pioneering work of Arjan van der Schaft and Bernhard Maschke on dissipative Hamiltonian systems including inputs and outputs \cite{van1981symmetries,MaschkeVdsBreedveld}. 
Unlike conservative Hamiltonian systems, pHS incorporate control and external factors into the dynamics through the ports.
The modelling approach also enables direct computation of the system output and Hamiltonian behaviour. 
Systems from various physical fields can be formulated as pHS, including thermodynamics, electromechanics, electromagnetics, fluid mechanics, hydrodynamics, or multi-body systems \cite{rashad2020twenty}.
Indeed, the functional structure of pHS, mitigating the modelling between conserved quantities, dissipation, input, and output, is a meaningful representation of many physical systems.

Recent studies point out that pHS frameworks are also relevant modelling approaches for multi-agent systems \cite{knorn2015overview}.  
Port-Hamiltonian multi-agent systems appear, for instance, in reliability engineering, for complex mechanical systems \cite{wang2016output,cristofaro2022fault}, consensus and opinion formation \cite{van2010Consensus,chang2014protocol,jafarian2015formation,WEI2017Consensus,xue2019opinion}, multi-input multi-output (MIMO) multi-agent systems \cite{sharf2019analysis}, swarm behaviour \cite{matei2019inferring,mavridis2020detection}, or autonomous vehicles, e.g., path-tracking \cite{ma2021path} or for modelling and safety analysis of adaptive cruise control systems \cite{knorn2014passivity,knorn2014scalability,dai2017safety,dai2020safety}. 
These microscopic agent-based modelling approaches are based on pHS using ordinary, stochastic, or delayed differential equations. 
Macroscopic traffic flow models \cite{bansal2021port} or, more generally, fluid dynamics models \cite{clemente2002geometric,rashad2021port,rashad2021portb} rely on infinite-dimensional pHS through partial differential equations using Stokes-Dirac bond graph representations.
In all the modelling approaches, the Hamiltonian quantifies the total energy in the system. 

In this article, we identify and analyse a general class of stochastic car-following models by second-order differential equations that can be formulated as stochastic input-state-output port-Hamiltonian systems.  
In contrast to classical car-following models, where the interaction depends only on the predecessor and is totally asymmetric, the vehicle dynamics in the pHS framework also depend on the distance to the follower. 
In Sec.~\ref{defMod}, we define the car-following model and its port-Hamiltonian formulation.
We determine sufficient stability conditions for the finite system with periodic boundary conditions in Sec.~\ref{stab} and point out that the stable system is ergodic and weakly converges to limiting solutions with unique invariant Gaussian measure. 
Some simulations are presented in Sec.~\ref{sim}. 
The presence of noise-induced waves is characterised in the second order using vehicle speed autocorrelation functions.
The simulation results show that the Hamiltonian component stabilises the system, reducing the total energy and the formation of stop-and-go dynamics.
Sec.~\ref{ccl} provides a summary of the results and some concluding remarks, in particular regarding the application of the model to adaptive cruise control systems.

\section{Stochastic port-Hamiltonian car-following model\label{defMod}}
\subsection{Notation}
In the following, we consider $N=3,4,\ldots$ vehicles on a segment of length $L$ with periodic boundary conditions (see Figure~\ref{fig:SystemScheme}). 
The periodic boundary conditions correspond to roundabout traffic, which has been widely studied experimentally in the literature \cite{Sugiyama2008,Stern2018}.
The periodic geometry facilitates the assessment of the global stability of the vehicle line. 
While local stability is concerned with individual platoons, global stability includes advective and convective perturbations that dissipate in a local context \cite{Orosz2009,Orosz2010}.
We denote by
\begin{itemize}
    \item $q(t)=\bigl(q_n(t)\bigr)_{n=1}^N\in\mathbb R^N$ the curvilinear positions of the vehicles and
    \item $p(t)=\bigl(p_n(t)\bigr)_{n=1}^N\in\mathbb R^N$ their momentum.
    \end{itemize}
Throughout this article, we assume that the vehicles have a normalised mass $m=1$ and, therefore, that momentum and speed coincide. 
We also assume that the vehicles are initially ordered by their indices, i.e.,
\begin{equation}
0\le q_1(0)\le q_2(0)\le\ldots\le q_N(0)\le L,
\end{equation}
and assume that the follower and predecessor of the $n$-th vehicle are the $(n-1)$-th and $(n+1)$-th vehicles, respectively, at any time. 
The predecessor of the $N$-th vehicle is the first vehicle and the follower of the first vehicle is the $N$-th vehicle due to the periodic boundaries.
The distances between the particles are the variables $Q(t)=\bigl(Q_n(t)\bigr)_{n=1}^N\in\mathbb R^N$ such that
\begin{equation}
\begin{cases}
~Q_n(t)=q_{n+1}(t)-q_n(t), \qquad n=1,\ldots,N-1,\\[1mm]
~Q_N(t)=L+q_1(t)-q_N(t).
\end{cases}
\end{equation}

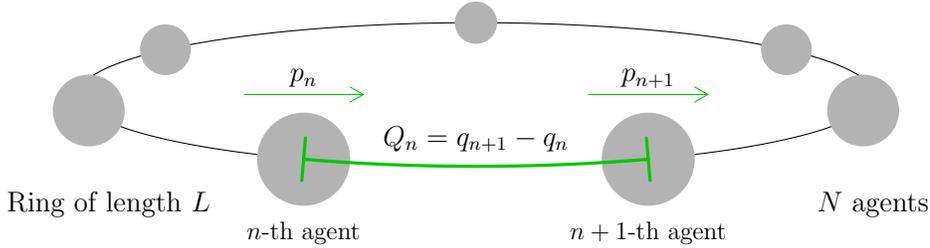
\begin{figure}[!ht]
\begin{center}\vspace{0mm}
    \input{Figures/fig1}\vspace{-1mm}
    \caption{Scheme of the single-file motion system with periodic boundary conditions. $q_n$ is the curvilinear position, while $Q_n = q_{n+1} - q_n$ is the spacing and $p_n$ is the momentum (speed) of the $n$-th vehicle.}
    \label{fig:SystemScheme}
\end{center}
\end{figure}

\subsection{Car-following model}
The microscopic motion model for the $n$-th vehicle is given by
\begin{equation}
   \begin{cases}
    ~dQ_n(t)=\bigl(p_{n+1}(t)-p_n(t)\bigr) dt,\\[2.5mm]
    ~dp_n(t)=\gamma \bigl(F(Q_n(t))-p_n(t)\bigr)dt+\beta \bigl(p_{n+1}(t)-p_n(t)\bigr)dt\\[1.5mm]
    \qquad\qquad\qquad\quad+\;\bigl(V'(Q_n(t))-V'(Q_{n-1}(t))\bigr)dt+\sigma dW_n(t),
    \end{cases}
    \label{modn}
\end{equation}
where $\gamma>0$, $\beta\ge0$ and assuming, according to the previous conventions, that $p_{n+1}$ is the speed of the first vehicle $p_1$ if $n=N$ and $Q_{n-1}$ is the spacing of the last vehicle $Q_{N}$ for $n=1$. 
The function $V\in C^1(\mathbb R,[0,\infty))$ is a convex distance-based interaction potential.
We use in the following the quadratic form 
\begin{equation}
V(x)=\frac\alpha2 x^2,\qquad \alpha\ge0,
\label{quadrapot}
\end{equation}
so that $V'(x)=\alpha x$. 
The Wiener processes $(W_n(t))_{n=1}^N\colon [0,\infty)\times \Omega \to \mathbb R^N$ are independent standard Brownian motions on the probability space $(\Omega,\mathcal{F},\mathbb{P})$. 
In the model, the vehicle acceleration combines three deterministic components:
\begin{enumerate}
    \item The first term $\gamma (F(Q_n)-p_n)$ relaxes the speed to a relationship $F:\mathbb R\mapsto\mathbb R_+$ between speed and spacing (parameter $\gamma>0$), usually called  in traffic engineering \textit{optimal velocity function} with reference to the \textit{optimal velocity model} by Bando et al.\ (1995) \cite{bando1995dynamical}, see \cite{tordeux2014collision,Orosz2009}.
    \item The second term $\beta (p_{n+1}-p_n)$ relaxes the speed of the considered vehicle to the speed of its predecessor (parameter $\beta\ge0$). The combination of drifts to a desired velocity relationship and to the speed of the predecessor has been first introduced by Helly in 1959 \cite{Helly1959} and later by Jiang et al.\ in 2001 with the \textit{full velocity difference model} \cite{jiang2001full}.
    \item The last term $V'(Q_n)-V'(Q_{n-1})$ uniforms the distance difference between the predecessor and the follower (potential $V$ with parameter $\alpha\ge0$); this term allows collision avoidance and is specific to the port-Hamiltonian car-following model.
\end{enumerate}
The optimal velocity function $F$ includes classical parameters of a car-following situation, such as the desired speed $v_0\ge0$, the size of the vehicle $\ell\ge0$, and the desired time gap $T>0$. A typical function is the piecewise linear form:
\begin{equation}
    F(x)=\min\Big\{v_0,\max\Big\{0,\frac{x-\ell}{T}\Big\}\Big\},
\end{equation}
which can be smoothed using mollifies, arctangent and other sigmoid functions.
Note that the homogeneous configurations where $Q_n=L/N$ and $p_n=F(L/N)$ for all $n\in\{1,\ldots,N\}$ are equilibrium solutions for the deterministic system where $\sigma=0$. 

The car-following model \eqref{modn} is a generalisation of the well-known \textit{optimal velocity model} by Bando et al.\ \cite{bando1995dynamical}, Helly's model \cite{Helly1959} and the \textit{full velocity difference model} by Jiang et al.\ \cite{jiang2001full}, and more recent stochastic models based on the Ornstein-Uhlenbeck process \cite{tordeux2016white,friesen2021spontaneous}.
In contrast to classical car-following models where the interaction depends only on the predecessor and is totally asymmetric, the dynamics also depend on the distance to the follower. 
This allows to obtain a skew symmetric structure, which is necessary for the port-Hamiltonian formulation.
The model remains a nearest-neighbour interaction model whose variables can be measured using sensors. 
This allows a direct application for autonomous adaptive cruise control systems. 

\subsection{Port-Hamiltonian formulation}
Recent results have shown that the two-dimensional Cucker-Smale swarm model can be represented as a port-Hamiltonian system \cite{matei2019inferring}. 
The formulation is conjectured for systems with $N$ agents and demonstrated with three interacting particles. 
It turns out that the one-dimensional car-following model \eqref{modn} for $N$ vehicles can also be formulated as a port-Hamiltonian system. 
The input port is an external speed control, providing energy to the system and direction of motion to the vehicles. 
This speed control depends on the spacing and acts as a dynamic feedback in the port-Hamiltonian dynamics.  
We write the periodic car-following system with $N$ vehicles as
\begin{equation}
    \left\{~\begin{aligned}
    dQ(t)&=Mp(t)dt\\
    dp(t)&=-M^\top V'(Q(t))dt+\beta Mp(t)dt + \gamma [F(Q(t))-p(t)]dt+\sigma dW(t),
    \end{aligned}\right.
    \label{mod}
\end{equation}
where $Q(t)=\bigl(Q_n(t)\bigr)_{n=1}^N\in\mathbb R^N$ and $p(t)=\bigl(p_n(t)\bigr)_{n=1}^N\in\mathbb R^N$ are the spacing and speed of the vehicles, respectively, $F(Q(t))=\bigl(F(Q_1(t)),\ldots,F(Q_N(t))\bigr)^\top$, while
$$
    M=\left[\begin{matrix}
    -1&1&\\
    &\ddots&\ddots&\\
    &&-1&1\\
    1&&&-1
    \end{matrix}\right]\in\mathbb R^{N\times N}~~\text{and}\quad W(t)=(W_n(t))_{n=1}^N\in\mathbb R^N.
$$
The Hamiltonian operator is given by
\begin{equation}
    H\bigl(Q(t),p(t)\bigr) = \sum_{n=1}^N V\bigl(Q_n(t)\bigr) + \frac12\sum_{n=1}^N p_n^2(t),
\end{equation}
so that
\begin{equation}
    \frac{\partial H}{\partial Q}=\left[\begin{matrix}
        V'(Q_1)\\\vdots\\V'(Q_N)
    \end{matrix}\right],
\end{equation}
while
\begin{equation}
    \frac{\partial H}{\partial p}=\left[\begin{matrix}
        p_1\\\vdots\\p_N
    \end{matrix}\right].
\end{equation}

\begin{proposition}\label{prop:PHSformulation}
Denoting $z(t)=[Q(t),p(t)]^\top\in\mathbb R^{2N}$, the port-Hamiltonian formulation of the car-following system \eqref{mod} reads
\begin{equation}\label{PHS}
\begin{cases}
    ~dz(t)=(J-R)\nabla H\bigl(z(t)\bigr)dt+G\bigl(z(t)\bigr)u \,dt+\sigma d\,\Xi(t),\\[1mm]
	~h(z(t)):=G^\top\bigl(z(t)\bigr)\nabla H\bigl(z(t)\bigr),\qquad\qquad\qquad z(0)=z_0,~~\Xi(0)=\Xi_0,
\end{cases}
\end{equation}
where
$$
   J=\left[\begin{matrix}
    0&M\\
    -M^\top&0
    \end{matrix}\right]\in \mathbb R^{2N\times 2N}, 
    \qquad
   R=\left[\begin{matrix}
    0&0\\
    0&\gamma I-\beta M
    \end{matrix}\right]\in \mathbb R^{2N\times 2N},
$$
$$
   G\bigl(z(t)\bigr)=\left[\begin{matrix}
    0\\
    \gamma \text{diag}\bigl(F(Q(t))\bigr)
    \end{matrix}\right]\in \mathbb R^{2N\times N},
    \qquad
    \Xi(t)=\left[\begin{matrix}
    0\\
    W(t)
    \end{matrix}\right]\in\mathbb R^{2N},
$$
and where $u=(1,\ldots,1)\in\mathbb R^{N}$ is the direction of motion of the vehicles.
Note that $J$ is skew symmetric by blocks $N\times N$ while, as shown in the following Lemma \ref{lemm:OpPostSemi}, $R$ is positive semi-definite. 
\end{proposition}

\begin{proof}
By construction, \eqref{mod} is recovered from \eqref{PHS}:
\begin{equation}
   dQ(t)=\Big[\big[0~~M\big]-\big[0~~0\big]\Big]{\left[\begin{matrix}
    V'(Q(t))\\p(t)\end{matrix}\right]}dt
    =Mp\,dt,
\end{equation}
while 
\begin{equation}
\begin{aligned}
   dp(t)&=\Big[\big[-M^\top~~0\big]-\big[0~~\gamma I-\beta M\big]\Big]{\left[\begin{matrix}
    V'(Q(t))\\p(t) \end{matrix}\right]}dt+\gamma F(Q(t))dt+\sigma dW(t)\\
    &=-M^\top V'(Q(t))dt+\beta Mp(t)dt+ \gamma[F(Q(t))-p(t)]dt+\sigma d W(t).
\end{aligned}\end{equation}
where $F(Q)=\bigl(F(Q_1),\ldots,F(Q_N)\bigr)^\top$.
\end{proof}

\begin{lemma}\label{lemm:OpPostSemi}
The operator $\gamma I - \beta M$ is positive semi-definite.
\end{lemma}

\begin{proof}
Observe that 
\[
x^\top(\gamma I-\beta M)x=(\gamma +\beta)\sum_{n=1}^N x_n^2-\beta\sum_{n=1}^Nx_{n+1}x_n,
\]
with the convention $x_{n+1}=x_1$ for $n=N$.

Now, observe that
\begin{equation}\label{equ:simple}
|xy| \leq \frac{1}{2} |x|^2 + \frac{1}{2}|y|^2.
\end{equation}
By using this inequality, it holds
\begin{align*}
& x^\top(-\beta M)x \\
&= \beta (x_1^2 + x_2^2 + \dots + x_n^2 - x_1 x_2 - x_2 x_3- \dots - x_n x_1)\\
&= \beta (\frac{1}{2} |x_1|^2 + \frac{1}{2}|x_2|^2 -x_1 x_2 + \frac{1}{2} |x_2|^2 + \frac{1}{2}|x_3|^2 -x_2 x_3+ \dots) \geq 0
\end{align*}
where at the last step we use \eqref{equ:simple}.

\end{proof}

The Hamiltonian structure holds thanks to the skew symmetric matrix $J$ and the Hamiltonian operator $H$. 
The speed difference terms of the car-following model are part of the dissipation matrix $R$ in the port-Hamiltonian formulation, while the \textit{optimal velocity} function takes the role of an input control matrix $G$ acting dynamically as feedback in the system \cite{ortega2001putting}. 
The resulting pHS is an input-state-output port-Hamiltonian system without interaction port \cite[Eq.~(29)]{van2006port}, with linear structure and dissipation components. 
In addition to the functional system modelling, control gain, and preservation of physical quantities, the technical advantages of pHS lie in the direct calculation of the Hamiltonian and system output by energy balance (see \eqref{EnergyBalance} below).
It is worth noting that the model is based on the speeds of and distances to the nearest neighbours, which can be measured autonomously using sensors. This allows direct applications for adaptive cruise control systems. 
Further remarks on the microscopic model and its port-Hamiltonian formulation are given below.

\begin{remark}
The \textit{optimal velocity model} of Bando et al.\ \cite{Bando1995} and the \textit{full velocity difference model} of Jiang et a. \cite{jiang2001full} have a port-Hamiltonian structure with a linear potential function $V(x)\propto x$ (see, e.g., \cite{knorn2014passivity}). 
In contrast, the port-Hamiltonian formulation Eq.~\eqref{PHS} of the car-following model Eq.~\eqref{modn} is valid for any convex potential function $V\in C^1(\mathbb R,[0,\infty))$.
\end{remark}

\begin{remark}
The model is purely Hamiltonian (the energy is conserved) if $\beta=\gamma=\sigma=0$. Depending on the initial conditions, the system can oscillate and describe limit cycles in this particular conserved deterministic case. If $\beta=\gamma=0$ and $\sigma>0$, the system, warmed by the noise and without dissipation, diverges. 
In addition, the model is Hamiltonian-dissipative if $\gamma=0$. The system is no longer port-Hamiltonian since the input matrix port $G$ and the output port $h$ are zero.
In general, as we will see in Sec.~\ref{stab}, some conditions on the model parameters are necessary for the system to be stable.
\end{remark}

\begin{remark}
For any $n\times n$ real (or complex) matrix $A$, one can form the matrix exponential \cite{teschl} as
\begin{equation}
	e^{t A}=\sum_{k=0}^{\infty} \frac{t^k}{k !} A^k.
\end{equation}
Therefore for any $t\in [0, T]$ the unique solution to the stochastic pHS system \eqref{PHS} (with Lipschitz coefficients) can be expressed with the Duhamel formula as \cite{DapratoZab92}
\begin{equation}
	z(t)=  e^{[J-R]t} z(0) + \int_{0}^{t} e^{[J-R](t-s)}G\bigl(z(s)\bigr)u\,ds + \sigma \int_{0}^{t} e^{[J-R](t-s)} d\,\Xi(s).
\end{equation}
\end{remark}

\begin{proposition}
The balance equation of the system describing the evolution in time of the Hamiltonian is given by
\begin{equation}
	dH(z)=\Big(\gamma\sum_{n=1}^N p_n(F(Q_n)-p_n)+\beta\sum_{n=1}^N p_n(p_{n+1}-p_n)+N\frac{\sigma^2}2\Big)dt+\sigma\sum_{n=1}^N p_ndW_n.
 \label{EnergyBalance}
\end{equation}
\end{proposition}
\begin{proof}
    We can write the infinitesimal generator of  Eq.~\eqref{PHS} using the Hamiltonian $H$ as potential function as
\begin{equation}
   \mathcal{L} H(z)=\nabla^\top\!H(z) h^\top(z)-\nabla^\top\!H(z)\, R\,\nabla H(z) 
   +\frac{\sigma^2}2 \text{Tr}\{\nabla^2 H(z)\}.
\end{equation}
While, using the Itô's formula, we obtain directly
\begin{equation}
	d H(z)=\mathcal{L} H(z) dt+\sigma\nabla^\top H(z) d\,\Xi.
\end{equation}
This allows to get, after simplifications, the balance equation
\begin{equation}
	dH(z)=\Big(\gamma\sum_{n=1}^N p_n(F(Q_n)-p_n)+\beta\sum_{n=1}^N p_n(p_{n+1}-p_n)+N\frac{\sigma^2}2\Big)dt+\sigma\sum_{n=1}^N p_ndW_n.
\end{equation}
\end{proof}
\begin{remark}
Note that the Hamiltonian behaviour \eqref{EnergyBalance} does not directly depend on the interaction potential and the related parameter $\alpha$ introduced in \eqref{quadrapot} due to the block skew symmetry of the matrix $J$. 
Furthermore, the first drift terms $\gamma\sum_{n=1}^N p_n(F(Q_n)-p_n)+\beta\sum_{n=1}^N p_n(p_{n+1}-p_n)$ are equal to zero at equilibrium where $Q_n=L/N$ and $p_n=F(L/N)$ for all $n\in\{1,\ldots,N\}$.
In addition, the stochastic noise induces a deterministic positive drift $N\sigma^2/2$ in the dynamics of the Hamiltonian. 
This is due to the convex shape of the Hamiltonian.
\end{remark}

The system representation using the spacing and speed variables $(Q,p)$ enables a compact port-Hamiltonian formulation.
In Sec.~\ref{stab}, we work out a different representation of the stochastic system and the solution using the position and speed variables $(q,p)$ to study the asymptotic properties of the system as $t\rightarrow\infty$.

\section{Stability analysis \label{stab}}

The port-Hamiltonian formulation \eqref{PHS} of the car-following model \eqref{modn} is valid for any convex distance-based interaction potential $V\in C^1(\mathbb R,[0,\infty))$.
Our main purpose in this section is to prove the stability of the stochastic system \eqref{mod} for the quadratic potential function 
\begin{equation}
V:x\mapsto V(x)=\alpha \frac{x^2}2, \qquad\alpha\ge0,
\end{equation}
and the affine \textit{optimal velocity} function 
\begin{equation}
F:x\mapsto F(x)=\frac{x-\ell}T, \qquad \ell\ge0, \quad T>0,
\end{equation}
where the system is linear.
The affine \textit{optimal velocity} function corresponds to a congested (interacting) traffic state where a driver regulates the speed according to the distance to the surrounding vehicles.

In order to recall and fix notation, let $W_n(t)$, $n=1,\ldots,N$, denote $N$ independent Wiener processes over the probability space $(\Omega, \mathcal{F}, \mathbb{P})$ on a filtered space satisfying the initial conditions.
We denote by $\mathcal{B}(\mathbb{R}^{2N})$ the Borel-$\sigma$ algebra over $\mathbb{R}^{2N}$ and by $\mathcal{M}_1(\mathbb{R}^{2N})$  the space of probability measures on $\mathbb{R}^{2N}.$ In the sequel, $C_b(\mathbb{R}^{2N})$ and $\operatorname{Lip}(\mathbb{R}^{2N})$ denote the space of bounded continuous and Lipschitz functions on $\mathbb{R}^{2N}$, respectively. 
For a function $\varphi \in C_b(\mathbb{R}^{2N})$ and $\mu \in \mathcal{M}_1(\mathbb{R}^{2N}),$ we denote the natural pairing in $\mathbb{R}^{2N}$ as
\begin{equation}
\langle\varphi, \mu\rangle:= \int_{\mathbb{R}^{2N}}^{} \varphi(x) \mu(dx).
\end{equation}

We first apply a change of coordinates and  write the system around the uniform configuration (perturbation system). For any $t\ge0$, the uniform configuration reads 
\begin{equation}\label{uniformconfig}
\left\{~\begin{aligned}
	    &q_n^{\mathcal H}(t)=q_n^{\mathcal H}(0)+v_{\mathcal H} t,  &&n=1,\ldots,N,\\
	&q_{n+1}^{\mathcal H}(t)-q_n^{\mathcal H}(t)=L / N, &&n=1,\ldots,N-1\\
    &L+q_{1}^{\mathcal H}(t)-q_n^{\mathcal H}(t)=L / N,
\end{aligned}\right.
\end{equation}
with $v_{\mathcal H}=F(L / N)$ the equilibrium speed of the system.
We expand the system around the uniform configuration as
\begin{equation}\label{uniformconfigtwo}
\begin{aligned}
	&x_n(t)=q_n(t)-q_n^{\mathcal H}(t) \\
	&y_n(t)=p_n(t)-v_{\mathcal H}
\end{aligned}
\end{equation}
For $n=1,\ldots,N,$ the stochastic system in the new coordinates is the linear homogeneous dynamics given for the $n$-th vehicle by
\begin{equation}\label{mainsystem}
   \begin{cases}
    ~dx_n(t)=y_n(t) dt,\\[2.5mm]
	~dy_n(t)= \gamma\left(\frac{1}{T}\big(x_{n+1}(t)-x_n(t)\big)-y_n(t)\right)dt +\beta\left(y_{n+1}(t)-y_n(t)\right)dt\\[1mm]
	\qquad\qquad~+\;\alpha\left(x_{n+1}(t)-2x_n(t)+x_{n-1}(t)\right)dt+\sigma dW_n(t).
\end{cases}
\end{equation}

Based on the parameters and for $k=1, \dots, N,$ we define the matrix $A$ through its entries as
\begin{equation}\label{matrixa}
	\begin{aligned}
		A_{pr} = \begin{cases}
			1 & \text{ if } p=2k-1,~ r=2k-1, \\
			-2\alpha -\gamma/T & \text{ if } p=2k,~ r=2k-1,\\
			\alpha & \text{ if } p=2k,~ r=2k-3,\\
			\alpha + \gamma/T & \text{ if } p=2k,~ r=2k+1,\\
			-(\beta+\gamma) & \text{ if } p=2k,~ r=2k,\\
			\beta & \text{ if } p=2k,~ r=2k+2,\\
			0 & \text{ otherwise, } 
		\end{cases}
	\end{aligned}
\end{equation}
following the cyclic ordering for negative row and column numbers.
Using the above constructions, we can succinctly express the stochastic system of equations as
\begin{equation}\label{mainmatrixsystem}
\begin{aligned}
	d {\mathcal{R}}(t)&=A \mathcal{R}(t) dt + \Lambda d\mathcal W(t)\\
	\mathcal{R}(0)&=x \in \mathbb{R}^{2N}
\end{aligned}
\end{equation}
where we define 
\begin{align*}
	\mathcal{R}(t) &=\left( x_{1}(t),  y_{1}(t),  x_{2}(t),  y_{2}(t), \ldots,  x_{N}(t),  y_{N}(t)\right)^\top,\\
	\mathcal W(t) &= \left(W_1(t), W_2(t), W_3(t), W_4(t), \ldots, W_{2N}(t)\right)^\top\\
	\Lambda &= \text{diag}(0,\sigma,0,\sigma,\ldots,0,\sigma)
\end{align*}
so that $\mathcal W(t)$ is the $2N$-dimensional Wiener process.
Accordingly, we define 
$$L^2(\Omega; \mathbb{R}^{2N}):= \{f:\Omega\rightarrow\mathbb{R}^{2N} \text{ measurable s.t. }  \mathbb{E}|f|^2 <\infty \}.$$ 

By using \cite[Lemma 10.9 \& Thm 10.10]{AA02} we obtain the following decomposition for the matrix $A$
\begin{equation}\label{decompositionOpA}
	A = P A P + (\text{id}-P) A (\text{id}-P) := A_0 + A_1
\end{equation}
where $P$ denotes the projection onto the zero eigenvalue space of $A$. We denote the spectral bound over the set of non-zero eigenvalues as
\begin{equation}\label{constantabar}
	\bar{a} = \underset{\lambda_k^{l}\ne0} {\min}	-\Re\big(\lambda^{(l)}_k\big)
\end{equation} 
with $\Re(u)$, $u\in\mathbb C$, the real part of the complex number $u$ and where  $\{\lambda^{(l)}_k\}$, $k=0,\ldots,N-1$, $l=1,2$ denote the $2N$ eigenvalues of $A$.
As we will see later in the section \ref{submata}, the $2N$ eigenvalues of $A$ are distinct, one is equal to zero while the remaining eigenvalues have non-positive real parts, so that $\bar a>0$. 

In this regard, the solution to \eqref{mainmatrixsystem} can be written as a sum of solution to
\begin{equation}\label{mainmatrixsystemtwo}
	d {\mathcal{R'}}(t)=A_1 \mathcal{R'}(t) dt + \Lambda d\mathcal W(t)
\end{equation}
plus a vector $m_0 \in P(\mathbb{R}^{2N}),$ namely a vector from the zero eigenspace, which leads to
\begin{equation}\label{decomposition}
	\mathcal{R} = \mathcal{R'} + m_0.
\end{equation}

When we denote by $\mathcal{R'}(t) = \mathcal{R'}(t,x)$ the solution to the stochastic system \eqref{mainmatrixsystemtwo} with initial condition $\mathcal{R'}(0)=x\in \mathbb{R}^{2N},$ the corresponding semigroup is defined as
\begin{align*}
P_t \varphi(x) := \mathbb{E}(\varphi (\mathcal{R'}(t,x)))
\end{align*}
for $\varphi \in \operatorname{Lip}(\mathbb{R}^{2N}).$ For $\Gamma \in \mathcal{B}(\mathbb{R}^{2N})$ and  $x \in \mathbb{R}^{2N},$ we also define the following transition probabilities $P_t(x, \cdot)$ as
\begin{equation}
	P_t(x, \Gamma):=\text{\textit{Law}}(\mathcal{R'}(t,x))(\Gamma) = \mathbb{P}\left[\mathcal{R'}(t,x) \in \Gamma\right]
\end{equation}
for all positive times $t>0.$
For $\mu \in \mathcal{M}_1(\mathbb{R}^{2N}),$  we further define
\begin{equation}\label{defpstar}
P_t^{\star}\mu(\Gamma) := \int P_t(x, \Gamma)\mu(dx), \  t \geq 0\ , \Gamma \in \mathcal{B}(\mathbb{R}^{2N})
\end{equation}
where recall that $\mathcal{M}_1(\mathbb{R}^{2N})$ denotes the space of probability measures on $\mathbb{R}^{2N}.$

Our main result in this section is the following which is proved in Subsection~\ref{substoch}. We refer the reader to \cite[Section 3.2]{pratoZabzyc96} for definition and characterization of ergodicity and the related background, which we also recall as needed in Subsection~\ref{substoch}.
\begin{theorem}\label{thmstoch}
	If the model parameters satisfy the condition
		\begin{equation}\label{stabilitycondition}
		\gamma>0,\qquad\gamma / 2+\beta+T \alpha > 1 / T,
	\end{equation}
the solution $\mathcal{R}(t, x)$ to the stochastic system \eqref{mainmatrixsystem}, as $t\rightarrow\infty,$ converges in law to a limiting solution 
$$\mathcal{R}_\infty =\left( x_{1}^\infty,  y_{1}^\infty,  x_{2}^\infty,  y_{2}^\infty, \ldots,  x_{N}^\infty,  y_{N}^\infty\right)^\top \in L^2(\Omega; \mathbb{R}^{2N}).$$

Moreover, the law $\mathcal N  \in \mathcal{M}_1(\mathbb{R}^{2N})$ of the limiting solution $\mathcal{R}_\infty$ is the unique ergodic invariant Gaussian measure $\mathcal N:= \mathcal N(0, \Sigma(\infty))$ where
	\begin{align*}
	\Sigma(t)&= \int_{0}^{t} e^{s A_1} \Lambda \Lambda^{\top} e^{s A_1^{\top}} d s. 
	\end{align*}
Furthermore, for $\varphi \in \operatorname{Lip}(\mathbb{R}^{2N})$ and $\mathcal N_t:= \mathcal N(0, \Sigma(t))$ the following estimate holds
\begin{equation}\label{expoestimate}
\left|\langle \varphi,\mathcal N_t \rangle-\langle\varphi, \mathcal N\rangle\right|^2 \leq \left(|x|^2+ \frac{N\sigma}{\bar{a}}\right)\|\varphi\|_{\text {Lip }}^2 e^{-2\bar{a} t},
\end{equation}
for all times $t\geq 0$, with $x$ the initial conditions and $\bar{a}>0$ the positive spectral bound as defined in \eqref{constantabar}.
\end{theorem}

\begin{remark}
We comment on the convergence \eqref{expoestimate}.  Observe that for small times $t\approx 0$ the bound on the difference between the limit and the solution depends on the parameters of the system:  larger variance, number of vehicles and size of the initial data increase the bound, whereas a larger decay parameter $\bar{a}$ decreases the bound.  For large times $t>\!\!>0$ there is an exponential decay with decay rate set by the parameter $\bar{a}.$
\end{remark}

\begin{remark}
Note that the convergence of the invariant measure  in \eqref{expoestimate} is a very useful form of convergence which directly implies also the convergence in Wasserstein distance.  The Wasserstein distance between the two  probability measures $\mathcal N$ and $\mathcal N_t$  on $\mathbb{R}^{2N}$ is equivalent to \cite[Section 2]{GS02}
\begin{equation}
	d_W(\mathcal N_t, \mathcal N):=\sup \left\{\left|\int \varphi(y) \mathcal N_t(dy)-\int \varphi(y) \mathcal N(dy) \right|:\|\varphi\|_{\text {Lip }} \leq 1\right\}.
\end{equation}
Since the supremum is taken over all $\|\varphi\|_{\text {Lip }} \leq 1,$ by using our estimate \eqref{expoestimate} it follows that
\begin{align*}
d_W(\mathcal N_t, \mathcal N) \leq \left(|x|^2+ \frac{N\sigma}{\bar{a}}\right) e^{-2\bar{a} t}
\end{align*}
which directly implies the convergence 
$$
d_W(\mathcal N_t, \mathcal N) \rightarrow 0
$$
as $t\rightarrow\infty.$
\end{remark}

\begin{remark}
The stability condition \eqref{stabilitycondition} covers the well-known stability condition of the \textit{optimal velocity model} \cite{bando1995dynamical}
$$
\gamma/2>1/T
$$
if $\alpha=\beta=0$ and the stability condition of the \textit{full velocity difference model} \cite{jiang2001full} 
$$
\gamma/2+\beta>1/T
$$
if $\alpha=0$. 
The Hamiltonian component, quantified by the parameter $\alpha\ge0$, allows to improve the stability. 
More precisely, setting $\alpha=1/T^2$ in the stability condition \eqref{stabilitycondition} allows stability requirement for any $T,\gamma>0$ or $\beta\ge0$. 
In general, the stability condition tends to hold when the relaxation rates $\alpha$, $\beta$, or $\gamma$ are high. 
However, high rates can lead to uncomfortable behaviour with strong accelerations. 
\end{remark}

In order to prove the theorem \ref{thmstoch}, we need the following spectral bound for the matrix $A$, as defined in \eqref{constantabar}, which is proved later in Subsection~\ref{submata}.

\begin{theorem}\label{thmmatrixa}
	Suppose that the matrix $A$ is diagonalisable and denote $\{\lambda^{(l)}_k\}$, $k=0,\ldots,N-1$, $l=1,2$ its $2N$ (distinct) eigenvalues. Suppose that the stability conditions \eqref{stabilitycondition} hold. 
	Then, it follows that
\begin{equation}
	\begin{aligned}
		\Re\big(\lambda^{(l)}_k\big)\leq 0,\qquad \text{for all }~k=0,\ldots,N-1,\text{ and all }~l=1,2
	\end{aligned}
\end{equation}
\end{theorem}

\subsection{Proof of Theorem \ref{thmstoch}: Stochastic system stability}\label{substoch}
In Theorem \ref{thmmatrixa} we obtained the condition for the eigenvalues of A
\begin{equation}
	\Re\big(\lambda^{(l)}_k\big)\leq 0
\end{equation}
for $k=0,\ldots,N-1$, $l=1,2.$  It also follows that the spectral bound $\bar{a}$ as in \eqref{constantabar} is well-defined and positive, that is $\bar{a}>0.$

For the operator $A_1,$ that we obtained through the decomposition \eqref{decompositionOpA}, we show below the following important estimates
\begin{equation}\label{equ:lumarphillips}
	\begin{aligned}
		& \Re \innerprod{A_1 q}{q} < -\overline{a} ||q||^{2} \\
		& ||e^{t A_1} q || \leq \sqrt{(2 N-1)} e^{-t\overline{a}} ||q||.
	\end{aligned}
\end{equation}
In order to show that the estimates \eqref{equ:lumarphillips} hold, we write the vector $q$ in the eigenbasis $\{{\Theta}_n\}_n$ of $A_1$ as $$
q=\sum_{n=1}^{2 N-1}\left\langle q, {\Theta}_n\right\rangle {\Theta}_n,
$$
which implies
\begin{align}
	&\Re \innerprod{A_1q}{q} \\
	&=  \Re  \innerprod{\sum_{n=1}^{2N-1} a_n \left\langle q, {\Theta}_n\right\rangle {\Theta}_n}{\sum_{n=1}^{2N-1} \left\langle q, {\Theta}_n\right\rangle {\Theta}_n}\\
	&= \sum_{n=1}^{2N-1}  \Re a_n \left\langle q, {\Theta}_n\right\rangle^{2}\leq -\overline{a} ||q||^{2}
\end{align}
For the proof of the second estimate in \eqref{equ:lumarphillips} we obtain similarly 
$$
\begin{aligned}
	\left\|e^{A_1 t} q\right\|_{}&=\left\|\sum_{n=1}^{2 N-1}\left\langle q, {\Theta}_n\right\rangle e^{a_n t} {\Theta}_n \right\| \\
	& \leq \sum_{n=1}^{2 N-1}\left|\left\langle q, {\Theta}_n\right\rangle\right| e^{\Re\left(a_n\right) t} \\
	& \leq \sqrt{(2 N-1)}\|q\|_{} e^{-\overline{a} t}.
\end{aligned}
$$

We now set out to prove that these estimates imply  the stochastic system  \eqref{mainmatrixsystemtwo} has a unique invariant measure which is also ergodic. By taking $q= x-y$ in the first estimate of \eqref{equ:lumarphillips} we obtain
\begin{equation}\label{equ:lumarphillips2}
	\begin{aligned}
		&  \innerprod{A_1(x-y)}{x-y} < -\bar{a}||x-y||^{2},
	\end{aligned}
\end{equation}
in the real Hilbert space setting.  

In order to show the existence of a unique ergodic invariant measure, we follow the classical arguments in \cite{pratoZabzyc96, DapratoZab92}.  For nonnegative times $t \geq 0$, we consider  Wiener processes $ V_n(t)$, independent of $ W_n(t)$, and  define

$$
\begin{aligned}
	&\overline{W}_n(t)=\left\{\begin{array}{l}
		W_n(t) \text { if } t \geq 0 \\
		V_n(-t) \text { if } t \leq 0.
	\end{array}\right.
\end{aligned}
$$

Next, for any $s\in\mathbb{R}$ and $x\in\mathbb{R}^{2N}$, we consider the following equation with the strong solution $\mathcal{R'}(t) = \mathcal{R'}(t, s, x),$ started at time $s$ for any $s\leq t,$ as
\begin{equation}\label{mainmatrixsystemthree}
	\begin{aligned}
		d {\mathcal{R'}}(t) &=A_1 \mathcal{R'}(t) dt + \Lambda d\overline{\mathcal W}(t)\\
		\mathcal{R'}(s) &= x,
	\end{aligned}
\end{equation}
where we defined
\begin{align*}
	\overline{\mathcal W}(t) &= \left(\overline W_1(t),\overline W_2(t), \overline W_3(t),\overline W_4(t), \ldots, \overline W_{2N}(t)\right)^\top.
\end{align*}

We apply Itô's lemma to the square of the strong solution $\left|\mathcal{R'}(t)\right|^2$, $t \geq s$. We calculate as
$$
\begin{aligned}
	d\left|\mathcal{R'}(t)\right|^2 =\left\{2\left\langle A_1 \mathcal{R'}(t), \mathcal{R'}(t) \right\rangle + \frac{N\sigma^2}{2} \right\} d t +2\left\langle \mathcal{R'}(t), \Lambda d \overline{\mathcal W}(t)\right\rangle,
\end{aligned}
$$
and obtain
$$
\begin{aligned}
	\frac{d}{d t} \mathbb{E}\left(\left|\mathcal{R'}(t)\right|^2\right) &=\mathbb{E}\left(2\left\langle A_1 \mathcal{R'}(t), \mathcal{R'}(t)\right\rangle+ \frac{N\sigma}{2}\right)  \leq-2\bar{a} \mathbb{E}\left(\left|\mathcal{R'}(t)\right|^2\right) + \frac{N\sigma}{2},
\end{aligned}
$$
where we used the estimate \eqref{equ:lumarphillips2}. From this, we obtain the estimate
$$
\mathbb{E}\left|\mathcal{R'}(t)\right|^2 \leq e^{-2\bar{a}(t-s) / 2} \left[|x|^2 - \frac{N\sigma}{2\bar{a}}\right] + \frac{N\sigma}{2\bar{a}},\qquad t \geq s .
$$
As $\bar{a}>0$ and $t\geq s,$ we can write
\begin{equation}\label{secondmomentestimate}
	\mathbb{E}|\mathcal{R'}(t)|^2 \leq \left(|x|^2+ \frac{N\sigma}{\bar{a}}\right),\qquad t \geq s.
\end{equation}

Now, for $\tau_1>\tau_0>0$ and $t \geq-\tau_0$ we define
$$
\mathcal{Z'}(t)=\mathcal{R'}(t,-\tau_0, x)-\mathcal{R'}(t,-\tau_1, x),
$$
where recall that $\mathcal{R'}(t,-\tau_0, x)$ and $\mathcal{R'}(t,-\tau_1, x)$ respectively denote the solutions to \eqref{mainmatrixsystemthree} started at times $-\tau_0$ and $-\tau_1$ with initial data $x\in \mathbb{R}^{2N}.$
This time we proceed similar to above with $|\mathcal{Z'}(t)|^2$ following the same steps and find the final estimate
$$
\mathbb{E}|\mathcal{Z'}(t)|^2 \leq e^{-2\bar{a}(t+\tau_0)} \mathbb{E}|\mathcal{R'}(-\tau_0,-\tau_1, x)|^2,\qquad t>-\tau_0,
$$
which we combine with \eqref{secondmomentestimate} to obtain
\begin{equation}\label{differencesecondmomentest}
	\mathbb{E}|\mathcal{R'}(0,-\tau_0, x)-\mathcal{R'}(0,-\tau_1, x)|^2 \leq \left(|x|^2+ \frac{N\sigma}{\bar{a}}\right) e^{-2\bar{a} \tau_0},\qquad \tau_1>\tau_0.
\end{equation} 
This implies convergence in law; we directly obtain
$$
\textit{Law}(\mathcal{R'}(t, 0, x))=\textit{Law}(\mathcal{R'}(0,-t, x)) \rightarrow \textit{Law}(\mathcal R_\infty)=\mathcal N \text {, weakly as } t \rightarrow+\infty,
$$
that is to say
\begin{equation}\label{equ:weakConv}
\langle P_t^{\star} \delta_x, \varphi\rangle \rightarrow \langle \mathcal N, \varphi\rangle \text { as } t \rightarrow+\infty 
\end{equation}
for $\varphi \in C_b(\mathbb{R}^{2N}).$

We want to show that the law $\mathcal N$  of $\mathcal R_\infty$ is the invariant measure satisfying the estimate \eqref{expoestimate}. In order to see this, we show that \eqref{equ:weakConv} implies $\mathcal N$ is an invariant measure for the stochastic system \eqref{mainmatrixsystemthree}.  For $\varphi \in C_b(\mathbb{R}^{2N})$ and the initial data $x \in \mathbb{R}^{2N}$ observe that $P_t^{\star} \delta_x =\textit{Law}(\mathcal R'(t, 0, x)),$ where $P_t^{\star}$ is as defined in \eqref{defpstar}.
For an arbitrary time $r>0$, we obtain
$$
\langle \varphi, P_{t+r}^{\star} \delta_x \rangle=\langle P_r \varphi, P_t^{\star} \delta_x\rangle.
$$
Since $P_r \varphi \in C_b(\mathbb{R}^{2N})$, we take the limit as $t \rightarrow \infty$ and use \eqref{equ:weakConv} to obtain
$$
\langle \varphi, \mathcal N \rangle=\langle P_r \varphi, \mathcal N\rangle=\langle\varphi, P_r^{\star} \mathcal N\rangle,
$$
which shows $\mathcal N=P_r^{\star} \mathcal N.$  So, we demonstrated that $\mathcal N$ is the unique invariant measure.

Finally, let $\varphi \in \operatorname{Lip}(\mathbb{R}^{2N})$, then by \eqref{differencesecondmomentest}
holds 
$$
\begin{aligned}
	\left|P_t \varphi(x)-P_s \varphi(x)\right|^2 &=|\mathbb{E}(\varphi(\mathcal{R'}(0,-t, x)))-\mathbb{E}(\varphi(\mathcal{R'}(0,-s, x)))|^2 \\
	& \leq\|\varphi\|_{\mathrm{Lip}}^2 \mathbb{E}|\mathcal{R'}(0,-t, x)-\mathcal{R'}(0,-s, x)|^2 \\
	& \leq \left(|x|^2+ \frac{N\sigma}{\bar{a}}\right)\|\varphi\|_{\mathrm{Lip}}^2 e^{-2\bar{a} t} .
\end{aligned}
$$
As $s \rightarrow+\infty$, we find
\begin{equation}\label{est:ptest}
\left|P_t \varphi(x)-\langle\varphi, \mathcal N\rangle\right|^2 \leq \left(|x|^2+ \frac{N\sigma}{\bar{a}}\right)\|\varphi\|_{\text {Lip }}^2 e^{-2\bar{a} t} .
\end{equation}

This clarifies the existence and uniqueness of unique invariant measure. Now, we can also explicitly calculate the limiting law.
The property $\mathbb{E}[\mathcal R'(t)] \rightarrow 0$ as $t\rightarrow \infty$ follows directly from the following Duhamel's formula for the system \eqref{mainmatrixsystemtwo}
\begin{align*}
	\mathcal R'(t)=e^{t A_1} \mathcal R'(0)+\int_{0}^{t} e^{(t-s) A_1} \Lambda d \mathcal W(s) 
\end{align*}
and the second part of the estimate \eqref{equ:lumarphillips}. We have that $\mathcal R'$ is a Gaussian random variable with expectation $\mu_{\mathcal R'}(t)$ and covariance operator $\Sigma(t)$  given by
\begin{equation}\label{def:variancesigmat}
\begin{aligned}
\mu_{\mathcal R'(0)}(t)=e^{t A_1} \mathcal R'(0), \quad \Sigma(t)=  \int_{0}^{t} e^{s A_1}  \Lambda \Lambda^{\top} e^{s A_1^{\top}} d s .
\end{aligned}
\end{equation}
 Accordingly, the characteristic function of $R'(t)$ reads
\begin{equation}
 \mathbb{E}\left[e^{i\langle p, \mathcal R'(t)\rangle}\right]=\exp \left(i\langle p, e^{t A_1} \mathcal R'(0) \rangle-\frac{1}{2}\langle p, \Sigma(t) p\rangle\right), \quad \forall p \in \mathbb{R}^{2 N}.
\end{equation}

By using the second estimate in \eqref{equ:lumarphillips}, it follows that $e^{t A_1} \mathcal R'(0) \rightarrow 0$ as $t\rightarrow \infty$. Therefore, in order to further examine the limiting behaviour, by Levy's continuity theorem \cite{Klenke04}, the asymptotic characteristic function is given by
\begin{equation}
	\lim _{t \rightarrow \infty} \mathbb{E}\left[e^{i\langle p, \mathcal R'(t)\rangle}\right]=\exp \left(-\frac{1}{2}\langle p, \Sigma(\infty) p\rangle\right), \quad \forall p \in \mathbb{R}^{2 N}
\end{equation}
for a well defined and finite $\Sigma(\infty).$  

We now need to show that $\Sigma(\infty)$ is finite.  We directly calculate
\begin{equation}
	\int_{0}^{\infty}\left|\left\langle p, e^{A_1 t} \Lambda \Lambda^{\top} e^{A_1^{\top} t} q\right\rangle\right| d t<\infty, \quad \forall p, q \in \mathbb{R}^{2 N}
\end{equation}
By using \eqref{equ:lumarphillips} one more time we obtain
\begin{equation}\label{expoEstimate}
	\begin{aligned}
		&\int_{0}^{\infty}\left|\left\langle p, e^{A_1 t} \Lambda \Lambda^{\top} e^{A_1^{\top} t} q\right\rangle\right| d t\\
		& \leq \int_{0}^{\infty}\left\| e^{A_1^{\top} t} p\right\|_{2 N}\left\|\Lambda \Lambda^{\top} e^{A_1^{\top} t} q\right\|_{2 N} d t \\
		& \leq\left(\int_{0}^{\infty}\left\|e^{A_1^{\top} t} p\right\|_{2 N}^{2} d t\right)^{1 / 2}\left(\int_{0}^{\infty}\left\|\Lambda \Lambda^{\top} e^{A_1^{\top} t} q\right\|_{2 N}^{2} d t\right)^{1 / 2}\\
		&\leq \left(\int_{0}^{\infty} e^{-\overline{a} t}  d t\right) ||\Lambda \Lambda^{\top}|| ||p|| ||q||\\
		&\leq \frac{||\Lambda \Lambda^{\top}|| ||p|| ||q||}{\overline{a}}
	\end{aligned}
\end{equation}
where we used the fact that the operator $A_1$ and  $A_1^{\top}$ have same eigenfunctions but conjugated corresponding eigenvalues.  In view of the decomposition \eqref{decomposition}, this leads to the invariant measure $\mathcal N= \mathcal N(0, \Sigma(\infty))$ since when we apply the coordinate transformation in \eqref{uniformconfig} the uniform configuration
\begin{equation}
(q_1^{\mathcal H}, v_{\mathcal H},2 L/N, v_{\mathcal H}, \dots, q_N^{\mathcal H}, v_{\mathcal H}) \in \mathbb{R}^{2N}
\end{equation}
is subtracted out.
In fact, the zero eigenvector of the matrix $A$ is zero ($m_0 =0$ in \eqref{decomposition}) in this new coordinates and in the space of physically interesting configurations with periodic boundary conditions that satisfies $\sum_{n=1}^{N} (x_{n+1}-x_n)= 0$.

In order to show the estimate \eqref{expoestimate}, by using the notation in \eqref{def:variancesigmat} with $\mathcal N_t= \mathcal N(0, \Sigma(t))$,  we observe
\begin{align*}
	P_t \varphi(x) = \mathbb{E}(\varphi (\mathcal{R'}(t,x))) = \int \varphi(y) \mathcal N_t(dy) = \langle \varphi,\mathcal N_t \rangle,
\end{align*}
for $\varphi \in \operatorname{Lip}(\mathbb{R}^{2N})$.
So that the estimate \eqref{expoestimate} directly follows from the estimate \eqref{est:ptest}.

In order to study ergodicity of $\mathcal N,$ we first recall the classical construction to realize the process $\overline{\mathcal W}(t)$ as a canonical coordinate process. For $E:=\mathbb{R}^{2N}, \mathcal{E}:=\mathcal{B}(\mathbb{R}^{2N}),$ by Daniell-Kolmogorov construction theorem \cite{KS91}, the process $\overline{\mathcal W}(t)$ can always be realized as a canonical coordinate process over the product space $(E^{\mathbb R}, \mathcal{E}^{\mathbb{R}})$ in the sense that
$$
\overline{\mathcal W}_t(\omega) = \omega_t 
$$
for $\omega \in E^{\mathbb R},$ where $E^{\mathbb R}$ denotes the functions $\mathbb{R} \rightarrow E$.   By defining a family of transformations $\theta_t: E^{\mathbb R} \rightarrow E^{\mathbb R}$
$$
(\theta_t\omega)(s):= \omega(t+s),
$$
we obtain a dynamical system $(E^{\mathbb R}, \mathcal{E}^{\mathbb{R}}, \theta_t, \mathbb{P}_{\mathcal{N}}),$ where $\mathcal{N}$ denotes the invariant measure we obtained before.  

In order to show that the invariant measure $\mathcal{N}$ is ergodic we argue by contradiction, following the standard arguments \cite{pratoZabzyc96}.  Suppose $\mathcal{N}$ is the unique invariant measure but there exists a measurable set $B \in \mathcal{B}(\mathbb{R}^{2N})$ with $\mathcal N(B) \in (0,1)$ such that
\begin{equation}\label{equ:contra}
P_{t} \mathbbm{1}_B =\mathbbm{1}_B 
\end{equation}
$\mathcal{N}-$almost surely, where $\mathbbm{1}_B$ denotes the characteristic function for the set $B$. 
For $ A \in \mathcal{B}(\mathbb{R}^{2N}),$ we define the measure
$$
\widetilde{\mathcal N}(A):=\frac{\mathcal{N}(A \cap B)}{\mathcal{N}(B)},
$$
which is the relative measure with respect to $B$.
We will check that the measure $\widetilde{\mathcal N}$ is also an invariant measure which is a contradiction to the uniqueness.

Observe that by using the assumption \eqref{equ:contra} we have 
$$P_{t}\left(x, A \cap B^{c}\right) \leq P_{t}\left(x, B^{c}\right)=0$$
$\mathcal{N}-$almost surely for $x \in B$ and
$$P_{t}(x, A \cap B) \leq P_{t}(x, B)=0$$ 
$\mathcal{N}-$almost surely for  $x \in B^{c}$. Therefore we obtain
\begin{align*}
\frac{1}{\mathcal{N}(B)} \int_{B^{c}} P_{t}(x, A \cap B) \mathcal{N}(d x)&=0\\
\frac{1}{\mathcal{N}(B)} \int_{B} P_{t}\left(x, A \cap B^{c}\right) \mathcal{N}(d x)&=0.
\end{align*} 
By using this, for arbitrary $t\geq 0$ and $A \in \mathcal{B}(\mathbb{R}^{2N})$, we can calculate
\begin{equation}
	\begin{aligned}
		&P_{t}^{*} \widetilde{\mathcal N}(A)=\int_{\mathbb{R}^{2N}} P_{t}(x, A) \widetilde{\mathcal N}(d x)=\frac{1}{\mathcal{N}(B)} \int_{B} P_{t}(x, A) \mathcal{N}(d x) \\
		&=\frac{1}{\mathcal{N}(B)} \int_{B} P_{t}(x, A \cap B) \mathcal{N}(d x)+\frac{1}{\mathcal{N}(B)} \int_{B} P_{t}\left(x, A \cap B^{c}\right) \mathcal{N}(d x) \\
		&=\frac{1}{\mathcal{N}(B)} \int_{B} P_{t}(x, A \cap B) \mathcal{N}(d x)+ \frac{1}{\mathcal{N}(B)} \int_{B^{c}} P_{t}(x, A \cap B) \mathcal{N}(d x)\\
		&+\frac{1}{\mathcal{N}(B)} \int_{B} P_{t}\left(x, A \cap B^{c}\right) \mathcal{N}(d x) \\
		& =\frac{1}{\mathcal{N}(B)} \int_{\mathbb{R}^{2N}} P_{t}(x, A \cap B) \mathcal{N}(d x).
	\end{aligned}
\end{equation}
 By using the invariance of $\mathcal{N}$ we further obtain
\begin{equation}
	\begin{aligned}
		P_{t}^{*} \widetilde{\mathcal N}(A) &=\frac{1}{\mathcal{N}(B)} \int_{\mathbb{R}^{2N}} P_{t}(x, A \cap B) \mathcal{N}(d x) \\
		&=\frac{1}{\mathcal{N}(B)} \mathcal{N}(A \cap B)=\widetilde{\mathcal N}(A),
	\end{aligned}
\end{equation}
which essentially implies that $\widetilde{\mathcal N}$ is another invariant measure.
This is a contradiction to the uniqueness of $\mathcal N$.  So, it follows that the unique invariant measure $\mathcal{N}$ is also ergodic. This completes the proof of Theorem \ref{thmstoch}.

\subsection{Proof of Theorem \ref{thmmatrixa}: Matrix $A$ spectral bound}\label{submata}
It is well known in the literature that there is a relationship between the spectral bound of a matrix and the stability of the corresponding dynamical system. In order to prove the spectral bound, we use this correspondence and a specifically constructed dynamical system. More precisely, we use the following proposition.  Recall that a linear system  is called stable if the solution remains bounded as $t \rightarrow \infty.$
\begin{proposition} (see \cite{teschl}).\label{teschlprop}
The linear system 
\begin{equation}\label{mainmatrixsystemdeter}
	d {\mathcal{S}}(t)=A \mathcal{S}(t) dt,
\end{equation}
where
\begin{align*}
	\mathcal{S}(t) :=\left( s_{1}(t),  s_{2}(t),\dots\dots, s_{M}(t)\right)^T
\end{align*}
is stable if and only if all eigenvalues $\Re\big(\lambda^{}_k\big)$ of $A$ satisfy $\Re\big(\lambda^{}_k\big)\leq 0$ and for all eigenvalues with $\Re\big(\lambda^{}_k\big)=0$ the corresponding algebraic and geometric multiplicities are equal.
\end{proposition}

For the proof of Theorem \ref{thmmatrixa}, we consider the solution $s(t)$ to the deterministic linear system
\begin{equation}\label{mainsystemUniform}
		\ddot{s}_n=\gamma\left(\frac{s_{n+1}-s_{n}}{T}-\dot{s}_n\right)+\beta\left(\dot{s}_{n+1}-\dot{s}_n\right)+\alpha\left(s_{n+1}-2s_n+s_{n-1}\right)
\end{equation}
using the exponential Ansatz 
$$
\begin{aligned}
	&s_n= \xi e^{\lambda t} e^{i n \theta_k}, \quad \dot{s}_n= \xi \lambda  e^{\lambda t} e^{i n \theta_k}, \quad \ddot{s}_n=\lambda^2\xi e^{\lambda t} e^{i n \theta_k}, \qquad\xi,\lambda\in\mathbb C,
\end{aligned}
$$
with $\theta_k=2\pi k/N$, $k=0,\ldots,N-1$. 
We obtain the following characteristic equation
\begin{equation}\label{characteristic}
	\begin{aligned}
		&\lambda^2+\gamma\left(\frac{1-e^{i \theta_k}}{T}+\lambda\right)+\beta \lambda\left(1-e^{i \theta_k}\right)+\alpha\left(2-e^{i \theta_k}-e^{-i \theta_k}\right)=0.
	\end{aligned}
\end{equation}
Note that we have two real solutions $\lambda_0^{(1)}=0$ and $\lambda_0^{(2)}=-\gamma$ for $\theta_k=0$ (i.e., $k=0$).
For simplicity, we put $c_\theta := \cos\theta, s_\theta :=\sin\theta$ and introduce the variables
\begin{equation}\label{equ:parameterab}
	\left\{ ~\begin{aligned}
		\mu_\theta &= \beta(1-c_\theta) + \gamma,\\      
		\sigma_\theta &= -\beta s_\theta,                              
	\end{aligned} \right. \qquad 
	\left\{ ~\begin{aligned}
		\nu_\theta &= (1-c_\theta) (\gamma/T+2\alpha),\\
		\rho_\theta &=  -s_\theta \gamma/T.
	\end{aligned} \right.
\end{equation}
Recall the well-known result \cite{F46} showing that the roots $\{\lambda_\theta\}_{\theta}$ of second order polynomial of the form $x^{2}+(\mu_{\theta}+\mathrm{i} \sigma_{\theta}) x +\nu_{\theta}+\mathrm{i} \rho_{\theta}=0$ satisfy $\Re(\lambda_\theta) <0$ if and only if we have the following conditions
\begin{equation}\label{determinantcondition} 
\mu_{\theta}>0\quad\text{and}\quad \operatorname{det}\left[\begin{array}{ccc}
		\mu_{\theta} & 0 & -\rho_{\theta} \\
		1 & \nu_{\theta} & -\sigma_{\theta} \\
		0 & \rho_{\theta} & \mu_{\theta}
	\end{array}\right]=\mu_{\theta}\left(\nu_{\theta} \mu_{\theta}+\rho_{\theta} \sigma_{\theta}\right)-\rho_{\theta}^{2}>0,
\end{equation}
see also \cite{Tordeux2012,tordeux2017influence,cordes2023single}. The parameters \eqref{equ:parameterab} with the second condition \eqref{determinantcondition} give the following condition for stability
\begin{equation}\label{equ:stability_general}
	\begin{aligned}
		(\beta(1-c_\theta)+\gamma) \left[(\beta(1-c_\theta)+\gamma)\left(1-c_\theta\right)\left(\frac{\gamma}{T}+2 \alpha\right)+ \frac{\beta \gamma}{T} s_\theta^2 \right] -\left(\frac{\gamma}{T}\right)^2 s_{\theta}^2 >0
	\end{aligned}
\end{equation}
For $x=c_\theta\in[-1,1]$ we write the l.h.t. of \eqref{equ:stability_general} in the following form
\begin{equation}
	f(x) = (1-x) g(x)
\end{equation}
where we defined
\begin{equation}
	\begin{aligned}
		g(x)=(\beta(1-x)+\gamma) \left[(\beta(1-x)+\gamma)\left(\frac{\gamma}{T}+2 \alpha\right)+ \frac{\beta \gamma}{T} (1+x) \right] -\left(\frac{\gamma}{T}\right)^2 (1+x). 
	\end{aligned}
\end{equation}
Note that $f(1)=0$ while $f'(x)=g'(x)(1-x)-g(x)$. 
Furthermore, the function $g$ is a second order polynomial with positive higher order coefficient if $\alpha,\beta>0$, while it is affine with negative slope if $\beta=0$ or $\alpha=0$.
From this, we can conclude by continuity that \eqref{determinantcondition}  holds if
\begin{equation}
	\begin{aligned}
		f'(1) = -g(1) = -\gamma \left[\frac{2\beta \gamma}{T}  + \frac{\gamma^2}{T} + 2 \gamma\alpha\right] + \frac{2\gamma^2}{T} <0
	\end{aligned}
\end{equation}
which, with $\gamma>0,$ simplifies to the stated stability condition given by 
\begin{equation}\label{stabilityconditiontwo}
	\gamma / 2+\beta+T \alpha > 1 / T.
\end{equation}

This implies that the specific system  
\begin{equation}\label{mainmatrixsystemdetertwo}
	d {\mathcal{S}}(t)=A \mathcal{S}(t) dt 
\end{equation}
with the matrix $A,$ as defined in \eqref{matrixa}, and
\begin{align*}
	\mathcal{S}(t) :=\left( s_{1}(t),  \dot s_{1}(t),  s_{2}(t),  \dot s_{2}(t), \ldots,  s_{N}(t),  \dot s_{N}(t)\right)^T
\end{align*}
is stable if $\gamma>0$ and if the condition \eqref{stabilityconditiontwo} is satisfied.     Then, the Proposition \ref{teschlprop} implies that the eigenvalues $\Re\big(\lambda^{(l)}_k\big)$ of the matrix $A$ necessarily satisfy the condition
\begin{equation}
	\begin{aligned}
		\Re\big(\lambda^{(l)}_k\big)\leq 0,\qquad\text{for all $k=0,\ldots,N-1$ and $l=1,2$}
	\end{aligned}
\end{equation}
as otherwise the solutions to this specifically constructed system would blow up too. This completes the proof of Theorem \ref{thmmatrixa}.

\begin{remark} Assuming the \textit{optimal velocity} function constant, i.e., $F(x)=v_{\mathcal H}$ (constant input control port), we obtain the characteristic equation
\begin{equation}\label{characteristictwo}
	\begin{aligned}
		&\lambda^2+\gamma\lambda+\beta \lambda\left(1-e^{i \theta_k}\right)+\alpha\left(2-e^{i \theta_k}-e^{-i \theta_k}\right)=0.
	\end{aligned}
\end{equation}
The model parameters read
\begin{equation}\label{equ:parameterabtwo}
	\left\{ ~\begin{aligned}
		\mu_\theta &= \beta(1-c_\theta) + \gamma,\\      
		\sigma_\theta &=-\beta s_\theta,                              
	\end{aligned} \right. \qquad 
	\left\{ ~\begin{aligned}
		\nu_\theta &=2\alpha(1-c_\theta),\\
		\rho_\theta &=0.
	\end{aligned} \right.
\end{equation}
The stability conditions \eqref{determinantcondition} are given by
\begin{equation}
    \beta(1-c_\theta) + \gamma>0\quad\text{and}\quad 2 \alpha\left(1-c_\theta\right)(\beta(1-c_\theta)+\gamma)^2>0,
\end{equation}
that holds for all $c_\theta\in[-1,1[$ if 
\begin{equation}
    \alpha(\beta+\gamma)>0.
\end{equation}
Here again, the Hamiltonian component $\alpha$ allows for system stabilisation.
\end{remark}

\section{Simulation results \label{sim}}

In the following, we present simulation results with fifty vehicles on a one kilometre long segment $L$ with periodic boundary conditions. 
The \textit{optimal velocity} function has the affine form 
$$
F(s)=\frac{s-\ell}T,
$$
with vehicle length $\ell=5$\;m and desired time gap $T=1$\;s. 
The relaxation rates are $\gamma=1$\;s$^{-1}$ and $\beta=0.5$\;s$^{-1}$. 
The noise amplitude is $\sigma=5$\;m\,s$^{-3/2}$.
The potential $V$ has the quadratic form 
$$
V(x)=\frac12\alpha x^2, \qquad \alpha\ge0.
$$
The simulations are carried out using an explicit/implicit Euler-Maruyama numerical scheme \cite{kloeden2011numerical}.
With $\delta t$ as the time step, the numerical scheme reads for all vehicles $n\in\{1,\ldots,N\}$
\begin{equation}
   \begin{cases}
    ~dq_n(t+\delta t)=q_(t)+\delta t p_n(t),\\[2.5mm]
    ~dp_n(t+\delta t)=p_n(t)+\delta t\gamma \left[\frac1T\big(q_{n+1}(t)-q_n(t)-\ell)\big)-p_n(t)\right]\\[1.5mm]
    \qquad+\;\delta t\left[\beta \big(p_{n+1}(t)-p_n(t)\big)+\alpha\big(Q_n(t)-Q_{n-1}(t)\big)\right]+\sqrt{\delta t}\,\xi_n(t),
    \end{cases}
    \label{modnsim}
\end{equation}
where $\xi_n(k\delta t)$, $k\in\mathbb N$, are independent, normal random variables.
We set the time step in the simulation to $\delta t=0.01$\;s . 
We repeat one hundred independent Monte Carlo simulations from uniform initial conditions and measure the system performance after a simulation time of 5e4 simulation steps. 
Such a setting allows the observation of stationary behaviour.
Different values for the Hamiltonian component $\alpha$ ranging from 0 to 1~s$^{-2}$ are tested.
Note that the stability condition 
$$
\gamma/2+\beta+\alpha T>1/T
$$
is critical if $\alpha=0$ and systematically holds as soon as $\alpha>0$. 

Figure~\ref{fig:Traj} shows illustrative examples of vehicle trajectories according to $\alpha\in\{0,0.05,0.1,0.2,0.5,1\}$\;s$^{-2}$. 
Although the system is systematically stable, we can observe stop-and-go dynamics when the parameter setting is close to the critical one (i.e., when $\alpha$ is close to zero). 
Interestingly, the Hamiltonian component, quantified by the parameter $\alpha$,  allows the system to stabilise. 
This demonstrates the benefits of using the distance to the follower in complement to the distance to the predecessor in the interaction. 
Note that similar behaviour can be obtained by increasing the other relaxation rates of the model $\beta$ or $\gamma$. 
However, high relaxation rates can lead to strong accelerations and uncomfortable behaviour. 
By taking into account the distance to the follower, the model offers a new degree of freedom for speed control in following situations. 
In addition, the Hamiltonian component does not affect the equilibrium velocity of the flow, allowing for increased flow stability without reducing the flow performance.

When the parameter $\alpha$, which determines the strength of the Hamiltonian, is high, there is a significant reduction in the total energy of the system.
The energy is measured thanks to the Hamiltonian $\tilde H$ of the perturbed system \eqref{uniformconfigtwo} given by
\begin{equation}
    \tilde H(t) := H(x(t),y(t)) = \frac12 \sum_{n=1}^N y_n^2(t) + \sum_{n=1}^N V(x_{n+1}(t)-x_n(t)),
\end{equation}
where $V(x)=\alpha x^2/2$, $\alpha\ge0$, is the quadratic potential. 
The energy $\tilde H$ tends to a minimum value as $\alpha$ increases, see Fig.~\ref{fig:Hamiltonian}. 
Some smoothing of the vehicle speed autocorrelation function can be observed in Fig.~\ref{fig:Autocorr}. 
This feature corroborates the dissipation of stop-and-go waves and energy in the dynamics for large $\alpha$ observed in Figs.~\ref{fig:Traj} and \ref{fig:Hamiltonian}.

\begin{figure}[!ht]
\begin{center}\medskip
    \SpaceTime{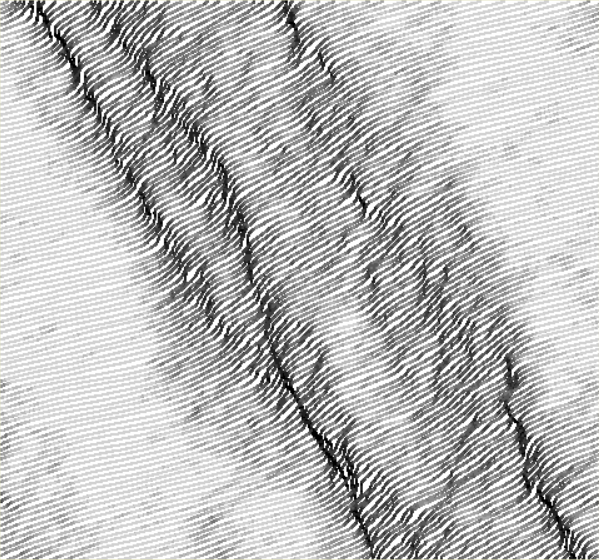}{0}{~}$\quad$\SpaceTime{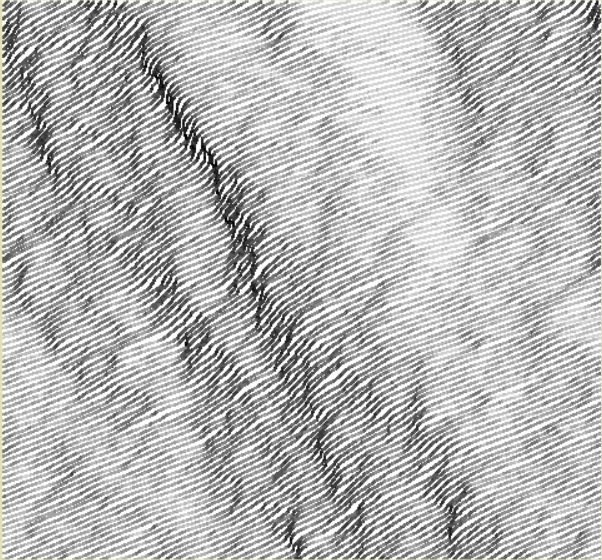}{0.05}{~s$^{-2}$}$\quad$\SpaceTime{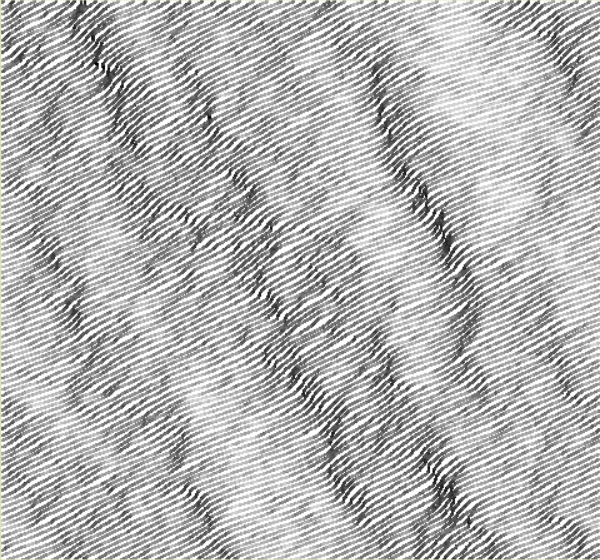}{0.1}{~s$^{-2}$}\\[7mm]\SpaceTime{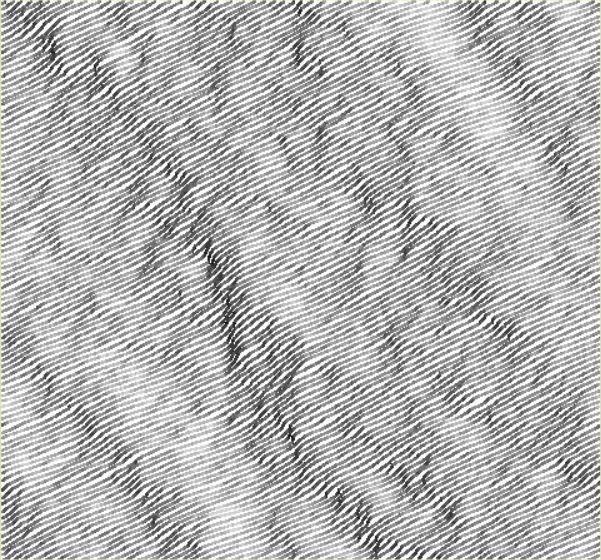}{0.2}{~s$^{-2}$}$\quad$\SpaceTime{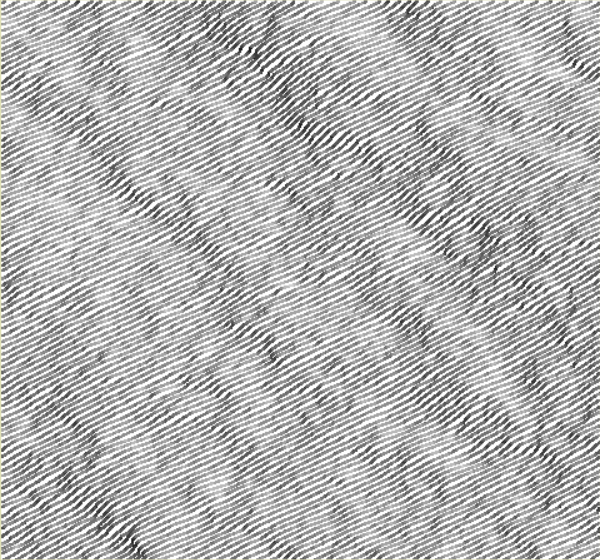}{0.5}{~s$^{-2}$}$\quad$\SpaceTime{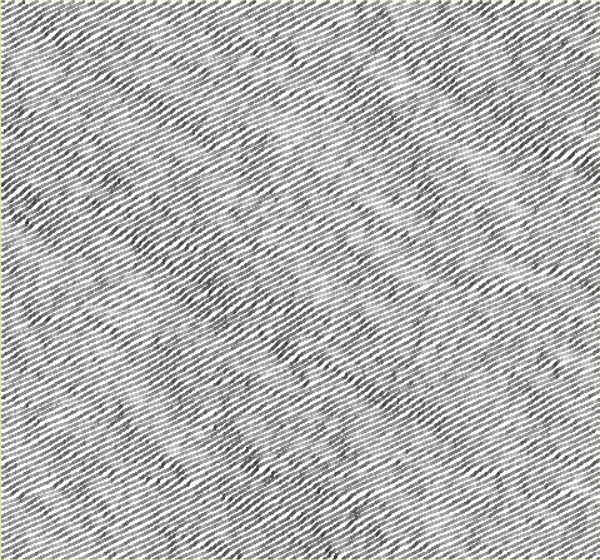}{1}{~s$^{-2}$}\bigskip
    \caption{Examples of trajectories over 120~s for 50 vehicles on a single-lane roundabout of 1 km length in stationary state, with $\alpha$ ranging between 0 and 1~s$^{-2}$. The grey intensity of the trajectories reflects the speed of the vehicles. Stop-and-go dynamics moving upstream dissipate as the Hamiltonian component $\alpha$ increases.}
    \label{fig:Traj}
\end{center}
\end{figure}

\begin{figure}[!ht]
\begin{center}\bigskip
    {\small\input{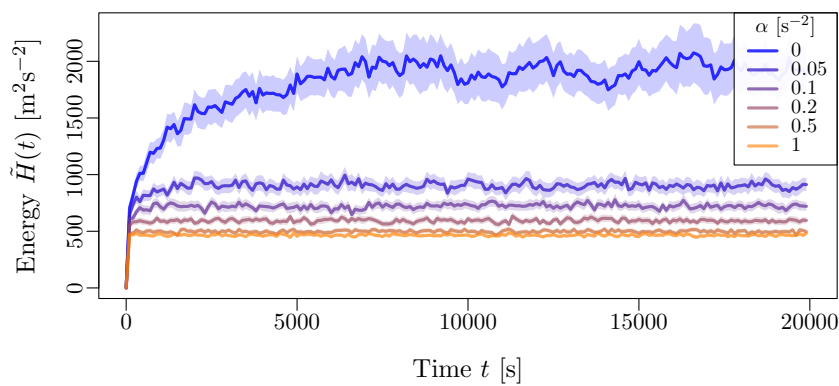}}
    \caption{Mean Hamiltonian behaviour with 95\% normal confidence interval averaged over 100 independent simulations for $\alpha$ between 0 and 1~s$^{-2}$. The total energy in the system decreases as $\alpha$ increases.}
    \label{fig:Hamiltonian}
\end{center}
\end{figure}

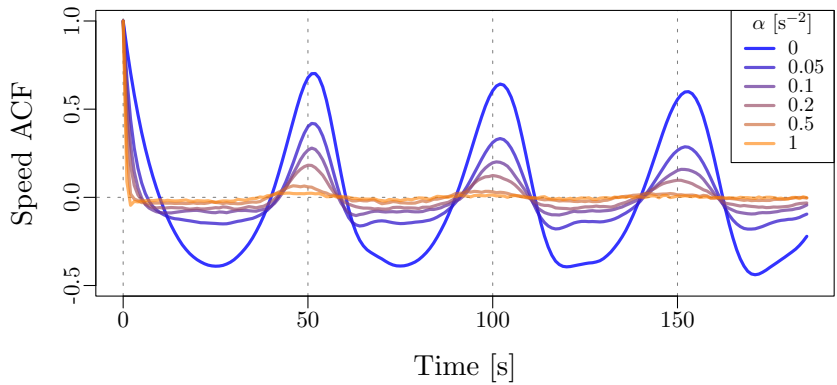
\begin{figure}[!ht]
\begin{center}\bigskip
    {\small\input{Figures/Autocorr_Speed}}
    \caption{Speed autocorrelation functions (ACF) for $\alpha$ ranging between 0 and 1~s$^{-2}$. Increasing the Hamiltonian component $\alpha$ leads to a smoothing of the speed ACF, confirming the dissipation of the stop-and-go dynamics.}
    \label{fig:Autocorr}
\end{center}
\end{figure}

\section{Summary and conclusion\label{ccl}}

We formulate in this article a general class of car-following models with a port-Hamiltonian structure. 
The model extends the well-known \textit{optimal velocity model} \cite{bando1995dynamical} and \textit{full velocity difference model} \cite{jiang2001full} by including the distance to the follower in the interaction. 
We establish sufficient stability conditions for the finite linear system with periodic solutions. 
The stability condition is general and covers the conditions of the classical models. 
We show that the stable system is ergodic and has a unique Gaussian invariant measure.
Simulation results point out that, even when the system is stable, it describes stop-and-go dynamics for parameter settings close to critical values. 
Interestingly, the Hamiltonian component stabilises the dynamics, confirming the advantages of using the distance to the follower as a complement to the distance to the predecessor in the interaction. 
In addition to the stabilisation, we observe a significant reduction in the energy of the system as the Hamiltonian component increases, as well as a smoothing of vehicle speed autocorrelation functions, attesting to the dissipation of the stop-and-go waves. 

By taking into account the distance to the follower, the port-Hamiltonian model provides a new degree of freedom to control the speed in following situations, which proves beneficial for the stabilisation of the system. 
The distance to the follower can be measured directly using radar or lidar sensors, making the modelling approach potentially applicable to autonomous cruise control systems. 
Furthermore, the Hamiltonian component does not affect the equilibrium velocity of the flow.
This feature allows for increased flow stability without reducing flow performance. 
System stability can be further improved using nonlinear interaction potentials $V$ and optimal velocity functions $F$. 
In this case, the skew symmetric and dissipation structures of the pHS remain linear. 
Only the input control matrix $G$ and the potential function $V$ are nonlinear. 
These investigations will be the subject of future work.

\subsection*{Acknowledgements}

We would like to thank Thomas Kruse for interesting discussions and useful comments.

\subsection*{Conflict of interest}

All authors declare no conflicts of interest in this paper.

\subsection*{References}

\bibliographystyle{unsrt}

\end{document}

%% file: Figures/fig1.tex
\begin{tikzpicture}[x=1.6pt,y=1.6pt]
\definecolor{fillColor}{RGB}{255,255,255}
\path[use as bounding box,fill=fillColor,fill opacity=0.00] (0,0) rectangle (247.08, 77.21);
\begin{scope}
\path[clip] ( 5.50, 5.50) rectangle (235.08, 65.21);
\definecolor{drawColor}{RGB}{0,0,0}

\path[draw=drawColor,line width= 0.4pt,line join=round,line cap=round] (123.54, 58.14) --
	(126.50, 58.14) --
	(129.46, 58.11) --
	(132.42, 58.07) --
	(135.36, 58.01) --
	(138.30, 57.93) --
	(141.22, 57.84) --
	(144.12, 57.73) --
	(147.00, 57.60) --
	(149.86, 57.46) --
	(152.69, 57.30) --
	(155.50, 57.13) --
	(158.27, 56.94) --
	(161.00, 56.73) --
	(163.70, 56.51) --
	(166.36, 56.27) --
	(168.98, 56.02) --
	(171.55, 55.75) --
	(174.07, 55.47) --
	(176.54, 55.18) --
	(178.96, 54.87) --
	(181.33, 54.54) --
	(183.64, 54.21) --
	(185.88, 53.86) --
	(188.07, 53.49) --
	(190.19, 53.12) --
	(192.24, 52.73) --
	(194.23, 52.34) --
	(196.14, 51.93) --
	(197.99, 51.51) --
	(199.76, 51.08) --
	(201.45, 50.63) --
	(203.06, 50.19) --
	(204.60, 49.73) --
	(206.05, 49.26) --
	(207.43, 48.78) --
	(208.72, 48.30) --
	(209.92, 47.81) --
	(211.04, 47.31) --
	(212.07, 46.81) --
	(213.01, 46.30) --
	(213.87, 45.79) --
	(214.63, 45.27) --
	(215.30, 44.75) --
	(215.88, 44.22) --
	(216.37, 43.69) --
	(216.77, 43.16) --
	(217.07, 42.63) --
	(217.28, 42.09) --
	(217.40, 41.56) --
	(217.42, 41.02) --
	(217.35, 40.48) --
	(217.19, 39.95) --
	(216.93, 39.41) --
	(216.58, 38.88) --
	(216.14, 38.35) --
	(215.60, 37.82) --
	(214.98, 37.30) --
	(214.26, 36.78) --
	(213.45, 36.26) --
	(212.55, 35.75) --
	(211.57, 35.25) --
	(210.49, 34.75) --
	(209.33, 34.25) --
	(208.08, 33.77) --
	(206.75, 33.29) --
	(205.34, 32.81) --
	(203.84, 32.35) --
	(202.27, 31.90) --
	(200.61, 31.45) --
	(198.88, 31.02) --
	(197.07, 30.59) --
	(195.20, 30.18) --
	(193.24, 29.77) --
	(191.22, 29.38) --
	(189.14, 29.00) --
	(186.98, 28.63) --
	(184.77, 28.27) --
	(182.49, 27.93) --
	(180.15, 27.60) --
	(177.76, 27.28) --
	(175.31, 26.98) --
	(172.82, 26.69) --
	(170.27, 26.42) --
	(167.67, 26.16) --
	(165.04, 25.91) --
	(162.36, 25.68) --
	(159.64, 25.47) --
	(156.88, 25.27) --
	(154.10, 25.09) --
	(151.28, 24.92) --
	(148.43, 24.77) --
	(145.56, 24.64) --
	(142.67, 24.52) --
	(139.76, 24.42) --
	(136.83, 24.34) --
	(133.89, 24.27) --
	(130.94, 24.22) --
	(127.98, 24.18) --
	(125.02, 24.17) --
	(122.06, 24.17) --
	(119.09, 24.18) --
	(116.14, 24.22) --
	(113.18, 24.27) --
	(110.24, 24.34) --
	(107.32, 24.42) --
	(104.40, 24.52) --
	(101.51, 24.64) --
	( 98.64, 24.77) --
	( 95.80, 24.92) --
	( 92.98, 25.09) --
	( 90.19, 25.27) --
	( 87.44, 25.47) --
	( 84.72, 25.68) --
	( 82.04, 25.91) --
	( 79.40, 26.16) --
	( 76.81, 26.42) --
	( 74.26, 26.69) --
	( 71.76, 26.98) --
	( 69.32, 27.28) --
	( 66.92, 27.60) --
	( 64.59, 27.93) --
	( 62.31, 28.27) --
	( 60.09, 28.63) --
	( 57.94, 29.00) --
	( 55.85, 29.38) --
	( 53.83, 29.77) --
	( 51.88, 30.18) --
	( 50.00, 30.59) --
	( 48.20, 31.02) --
	( 46.47, 31.45) --
	( 44.81, 31.90) --
	( 43.24, 32.35) --
	( 41.74, 32.81) --
	( 40.33, 33.29) --
	( 38.99, 33.77) --
	( 37.75, 34.25) --
	( 36.59, 34.75) --
	( 35.51, 35.25) --
	( 34.52, 35.75) --
	( 33.63, 36.26) --
	( 32.82, 36.78) --
	( 32.10, 37.30) --
	( 31.47, 37.82) --
	( 30.94, 38.35) --
	( 30.49, 38.88) --
	( 30.14, 39.41) --
	( 29.89, 39.95) --
	( 29.72, 40.48) --
	( 29.65, 41.02) --
	( 29.68, 41.56) --
	( 29.79, 42.09) --
	( 30.00, 42.63) --
	( 30.31, 43.16) --
	( 30.70, 43.69) --
	( 31.19, 44.22) --
	( 31.78, 44.75) --
	( 32.45, 45.27) --
	( 33.21, 45.79) --
	( 34.06, 46.30) --
	( 35.01, 46.81) --
	( 36.04, 47.31) --
	( 37.16, 47.81) --
	( 38.36, 48.30) --
	( 39.65, 48.78) --
	( 41.02, 49.26) --
	( 42.48, 49.73) --
	( 44.01, 50.19) --
	( 45.63, 50.63) --
	( 47.32, 51.08) --
	( 49.09, 51.51) --
	( 50.93, 51.93) --
	( 52.85, 52.34) --
	( 54.83, 52.73) --
	( 56.89, 53.12) --
	( 59.01, 53.49) --
	( 61.19, 53.86) --
	( 63.44, 54.21) --
	( 65.75, 54.54) --
	( 68.11, 54.87) --
	( 70.53, 55.18) --
	( 73.01, 55.47) --
	( 75.53, 55.75) --
	( 78.10, 56.02) --
	( 80.72, 56.27) --
	( 83.38, 56.51) --
	( 86.07, 56.73) --
	( 88.81, 56.94) --
	( 91.58, 57.13) --
	( 94.38, 57.30) --
	( 97.22, 57.46) --
	(100.07, 57.60) --
	(102.96, 57.73) --
	(105.86, 57.84) --
	(108.78, 57.93) --
	(111.71, 58.01) --
	(114.66, 58.07) --
	(117.61, 58.11) --
	(120.57, 58.14) --
	(123.54, 58.14);
\definecolor{fillColor}{gray}{0.70}

\path[fill=fillColor] ( 82.80, 25.85) circle ( 11);

\path[fill=fillColor] (164.27, 25.85) circle ( 11);

\path[fill=fillColor] (215.07, 37.37) circle ( 8.5);

\path[fill=fillColor] (196.94, 51.75) circle (  6);

\path[fill=fillColor] (123.54, 58.14) circle (  5);

\path[fill=fillColor] ( 50.13, 51.75) circle (  6);

\path[fill=fillColor] ( 32.00, 37.37) circle ( 8.5);
\definecolor{drawColor}{RGB}{0,200,0}

\path[draw=drawColor,line width= 1.2pt,line join=round,line cap=round] ( 82.80, 25.85) --
	( 83.57, 25.78) --
	( 84.34, 25.72) --
	( 85.12, 25.65) --
	( 85.90, 25.59) --
	( 86.68, 25.53) --
	( 87.46, 25.47) --
	( 88.25, 25.41) --
	( 89.04, 25.35) --
	( 89.83, 25.30) --
	( 90.63, 25.24) --
	( 91.43, 25.19) --
	( 92.23, 25.14) --
	( 93.03, 25.09) --
	( 93.84, 25.04) --
	( 94.65, 24.99) --
	( 95.46, 24.94) --
	( 96.27, 24.90) --
	( 97.09, 24.85) --
	( 97.91, 24.81) --
	( 98.73, 24.77) --
	( 99.55, 24.73) --
	(100.37, 24.69) --
	(101.20, 24.65) --
	(102.02, 24.62) --
	(102.85, 24.58) --
	(103.69, 24.55) --
	(104.52, 24.52) --
	(105.35, 24.49) --
	(106.19, 24.46) --
	(107.03, 24.43) --
	(107.86, 24.40) --
	(108.70, 24.38) --
	(109.55, 24.35) --
	(110.39, 24.33) --
	(111.23, 24.31) --
	(112.08, 24.29) --
	(112.92, 24.27) --
	(113.77, 24.26) --
	(114.61, 24.24) --
	(115.46, 24.23) --
	(116.31, 24.22) --
	(117.16, 24.20) --
	(118.01, 24.19) --
	(118.86, 24.19) --
	(119.71, 24.18) --
	(120.56, 24.17) --
	(121.41, 24.17) --
	(122.26, 24.17) --
	(123.11, 24.16) --
	(123.96, 24.16) --
	(124.82, 24.17) --
	(125.67, 24.17) --
	(126.52, 24.17) --
	(127.37, 24.18) --
	(128.22, 24.19) --
	(129.07, 24.19) --
	(129.92, 24.20) --
	(130.77, 24.22) --
	(131.62, 24.23) --
	(132.46, 24.24) --
	(133.31, 24.26) --
	(134.16, 24.27) --
	(135.00, 24.29) --
	(135.85, 24.31) --
	(136.69, 24.33) --
	(137.53, 24.35) --
	(138.37, 24.38) --
	(139.21, 24.40) --
	(140.05, 24.43) --
	(140.89, 24.46) --
	(141.72, 24.49) --
	(142.56, 24.52) --
	(143.39, 24.55) --
	(144.22, 24.58) --
	(145.05, 24.62) --
	(145.88, 24.65) --
	(146.71, 24.69) --
	(147.53, 24.73) --
	(148.35, 24.77) --
	(149.17, 24.81) --
	(149.99, 24.85) --
	(150.81, 24.90) --
	(151.62, 24.94) --
	(152.43, 24.99) --
	(153.24, 25.04) --
	(154.04, 25.09) --
	(154.85, 25.14) --
	(155.65, 25.19) --
	(156.45, 25.24) --
	(157.24, 25.30) --
	(158.04, 25.35) --
	(158.83, 25.41) --
	(159.62, 25.47) --
	(160.40, 25.53) --
	(161.18, 25.59) --
	(161.96, 25.65) --
	(162.73, 25.72) --
	(163.51, 25.78) --
	(164.27, 25.85);

\path[draw=drawColor,line width= 1.2pt,line join=round,line cap=round] ( 82.80, 25.85) -- ( 83.57, 25.78);

\path[draw=drawColor,line width= 1.2pt,line join=round,line cap=round] ( 83.24, 30.89) --
	( 82.80, 25.85) --
	( 82.37, 20.81);

\path[draw=drawColor,line width= 1.2pt,line join=round,line cap=round] (164.27, 25.85) -- (163.51, 25.78);

\path[draw=drawColor,line width= 1.2pt,line join=round,line cap=round] (164.71, 20.81) --
	(164.27, 25.85) --
	(163.84, 30.89);
\definecolor{drawColor}{RGB}{0,0,0}


\definecolor{drawColor}{RGB}{0,200,0}

\path[draw=drawColor,line width= 0.4pt,line join=round,line cap=round] ( 68.72, 41.15) -- ( 96.89, 41.15);

\path[draw=drawColor,line width= 0.4pt,line join=round,line cap=round] ( 93.76, 39.35) --
	( 96.89, 41.15) --
	( 93.76, 42.96);
\definecolor{drawColor}{RGB}{0,0,0}

\node[text=drawColor,anchor=base,inner sep=0pt, outer sep=0pt, scale=  0.90] at ( 82.80, 44.25) {$p_n$};

\definecolor{drawColor}{RGB}{0,200,0}

\path[draw=drawColor,line width= 0.4pt,line join=round,line cap=round] (150.19, 41.15) -- (178.36, 41.15);

\path[draw=drawColor,line width= 0.4pt,line join=round,line cap=round] (175.23, 39.35) --
	(178.36, 41.15) --
	(175.23, 42.96);
\definecolor{drawColor}{RGB}{0,0,0}

\node[text=drawColor,anchor=base,inner sep=0pt, outer sep=0pt, scale=  0.90] at (164.27, 44.25) {$p_{n+1}$};

\node[text=drawColor,anchor=base,inner sep=0pt, outer sep=0pt, scale=  0.90] at (123.54, 29) {$Q_n=q_{n+1}-q_n$};

\node[text=drawColor,anchor=base,inner sep=0pt, outer sep=0pt, scale=  0.90] at ( 36.69, 13.45) {Ring of length $L$};

\node[text=drawColor,anchor=base,inner sep=0pt, outer sep=0pt, scale=  0.90] at (217.43, 13.45) {$N$ agents};

\node[text=drawColor,anchor=base,inner sep=0pt, outer sep=0pt, scale=  0.80] at ( 82.80, 7) {$n$-th agent};

\node[text=drawColor,anchor=base,inner sep=0pt, outer sep=0pt, scale=  0.80] at (164.27, 7) {$n+1$-th agent};

\end{scope}
\end{tikzpicture}

%% file: Figures/Autocorr_Speed.tex
\begin{tikzpicture}[x=.75pt,y=.75pt]
\definecolor{fillColor}{RGB}{255,255,255}
\path[use as bounding box,fill=fillColor,fill opacity=0.00] (0,0) rectangle (432.38,216.19);
\begin{scope}
\path[clip] (  0.00,  0.00) rectangle (432.38,216.19);
\definecolor{drawColor}{RGB}{0,0,0}

\path[draw=drawColor,line width= 0.4pt,line join=round,line cap=round] ( 61.79, 48.00) -- (341.36, 48.00);

\path[draw=drawColor,line width= 0.4pt,line join=round,line cap=round] ( 61.79, 48.00) -- ( 61.79, 42.00);

\path[draw=drawColor,line width= 0.4pt,line join=round,line cap=round] (154.98, 48.00) -- (154.98, 42.00);

\path[draw=drawColor,line width= 0.4pt,line join=round,line cap=round] (248.17, 48.00) -- (248.17, 42.00);

\path[draw=drawColor,line width= 0.4pt,line join=round,line cap=round] (341.36, 48.00) -- (341.36, 42.00);

\node[text=drawColor,anchor=base,inner sep=0pt, outer sep=0pt, scale=  0.80] at ( 61.79, 32.40) {0};

\node[text=drawColor,anchor=base,inner sep=0pt, outer sep=0pt, scale=  0.80] at (154.98, 32.40) {50};

\node[text=drawColor,anchor=base,inner sep=0pt, outer sep=0pt, scale=  0.80] at (248.17, 32.40) {100};

\node[text=drawColor,anchor=base,inner sep=0pt, outer sep=0pt, scale=  0.80] at (341.36, 32.40) {150};

\path[draw=drawColor,line width= 0.4pt,line join=round,line cap=round] ( 48.00, 53.34) -- ( 48.00,186.85);

\path[draw=drawColor,line width= 0.4pt,line join=round,line cap=round] ( 48.00, 53.34) -- ( 42.00, 53.34);

\path[draw=drawColor,line width= 0.4pt,line join=round,line cap=round] ( 48.00, 97.84) -- ( 42.00, 97.84);

\path[draw=drawColor,line width= 0.4pt,line join=round,line cap=round] ( 48.00,142.35) -- ( 42.00,142.35);

\path[draw=drawColor,line width= 0.4pt,line join=round,line cap=round] ( 48.00,186.85) -- ( 42.00,186.85);

\node[text=drawColor,rotate= 90.00,anchor=base,inner sep=0pt, outer sep=0pt, scale=  0.80] at ( 39.60, 53.34) {-0.5};

\node[text=drawColor,rotate= 90.00,anchor=base,inner sep=0pt, outer sep=0pt, scale=  0.80] at ( 39.60, 97.84) {0.0};

\node[text=drawColor,rotate= 90.00,anchor=base,inner sep=0pt, outer sep=0pt, scale=  0.80] at ( 39.60,142.35) {0.5};

\node[text=drawColor,rotate= 90.00,anchor=base,inner sep=0pt, outer sep=0pt, scale=  0.80] at ( 39.60,186.85) {1.0};

\path[draw=drawColor,line width= 0.4pt,line join=round,line cap=round] ( 48.00, 48.00) --
	(420.38, 48.00) --
	(420.38,192.19) --
	( 48.00,192.19) --
	cycle;
\end{scope}
\begin{scope}
\path[clip] (  0.00,  0.00) rectangle (432.38,216.19);
\definecolor{drawColor}{RGB}{0,0,0}

\node[text=drawColor,anchor=base,inner sep=0pt, outer sep=0pt, scale=  1.00] at (234.19,  8.40) {Time [s]};

\node[text=drawColor,rotate= 90.00,anchor=base,inner sep=0pt, outer sep=0pt, scale=  1.00] at ( 15.60,120.10) {Speed ACF};
\end{scope}
\begin{scope}
\path[clip] ( 48.00, 48.00) rectangle (420.38,192.19);
\definecolor{drawColor}{RGB}{0,0,0}

\path[draw=drawColor,draw opacity=0.50,line width= 0.4pt,dash pattern=on 1pt off 3pt ,line join=round,line cap=round] ( 43.15, 97.84) -- (406.59, 97.84);

\path[draw=drawColor,draw opacity=0.50,line width= 0.4pt,dash pattern=on 1pt off 3pt ,line join=round,line cap=round] ( 61.79, 48.00) -- ( 61.79,192.19);

\path[draw=drawColor,draw opacity=0.50,line width= 0.4pt,dash pattern=on 1pt off 3pt ,line join=round,line cap=round] (154.98, 48.00) -- (154.98,192.19);

\path[draw=drawColor,draw opacity=0.50,line width= 0.4pt,dash pattern=on 1pt off 3pt ,line join=round,line cap=round] (248.17, 48.00) -- (248.17,192.19);

\path[draw=drawColor,draw opacity=0.50,line width= 0.4pt,dash pattern=on 1pt off 3pt ,line join=round,line cap=round] (341.36, 48.00) -- (341.36,192.19);
\definecolor{drawColor}{RGB}{0,0,255}

\path[draw=drawColor,draw opacity=0.80,line width= 1.2pt,line join=round,line cap=round] ( 61.79,186.85) --
	( 63.66,173.07) --
	( 65.52,160.78) --
	( 67.38,149.68) --
	( 69.25,139.69) --
	( 71.11,130.58) --
	( 72.97,122.48) --
	( 74.84,115.19) --
	( 76.70,108.59) --
	( 78.57,102.61) --
	( 80.43, 97.13) --
	( 82.29, 92.16) --
	( 84.16, 87.55) --
	( 86.02, 83.59) --
	( 87.89, 80.02) --
	( 89.75, 76.79) --
	( 91.61, 73.91) --
	( 93.48, 71.55) --
	( 95.34, 69.46) --
	( 97.20, 67.80) --
	( 99.07, 66.37) --
	(100.93, 65.25) --
	(102.80, 64.31) --
	(104.66, 63.62) --
	(106.52, 63.24) --
	(108.39, 63.08) --
	(110.25, 63.16) --
	(112.11, 63.49) --
	(113.98, 64.17) --
	(115.84, 65.09) --
	(117.71, 66.35) --
	(119.57, 68.00) --
	(121.43, 69.87) --
	(123.30, 72.01) --
	(125.16, 74.42) --
	(127.02, 77.26) --
	(128.89, 80.60) --
	(130.75, 84.23) --
	(132.62, 88.30) --
	(134.48, 92.83) --
	(136.34, 97.75) --
	(138.21,103.39) --
	(140.07,109.50) --
	(141.93,115.91) --
	(143.80,122.66) --
	(145.66,129.55) --
	(147.53,136.39) --
	(149.39,143.11) --
	(151.25,149.29) --
	(153.12,154.39) --
	(154.98,158.23) --
	(156.85,160.26) --
	(158.71,160.34) --
	(160.57,158.17) --
	(162.44,153.62) --
	(164.30,146.90) --
	(166.16,138.52) --
	(168.03,129.01) --
	(169.89,119.12) --
	(171.76,109.31) --
	(173.62,100.11) --
	(175.48, 92.08) --
	(177.35, 85.46) --
	(179.21, 80.18) --
	(181.07, 76.21) --
	(182.94, 73.37) --
	(184.80, 71.37) --
	(186.67, 70.04) --
	(188.53, 68.91) --
	(190.39, 67.80) --
	(192.26, 66.78) --
	(194.12, 65.58) --
	(195.98, 64.51) --
	(197.85, 63.72) --
	(199.71, 63.27) --
	(201.58, 63.18) --
	(203.44, 63.29) --
	(205.30, 63.70) --
	(207.17, 64.38) --
	(209.03, 65.22) --
	(210.89, 66.42) --
	(212.76, 68.01) --
	(214.62, 69.86) --
	(216.49, 72.10) --
	(218.35, 74.75) --
	(220.21, 77.81) --
	(222.08, 81.20) --
	(223.94, 84.86) --
	(225.81, 88.95) --
	(227.67, 93.32) --
	(229.53, 98.18) --
	(231.40,103.30) --
	(233.26,108.87) --
	(235.12,114.69) --
	(236.99,120.68) --
	(238.85,126.76) --
	(240.72,132.75) --
	(242.58,138.46) --
	(244.44,143.76) --
	(246.31,148.21) --
	(248.17,151.70) --
	(250.03,154.03) --
	(251.90,155.04) --
	(253.76,154.48) --
	(255.63,152.24) --
	(257.49,148.42) --
	(259.35,143.00) --
	(261.22,136.05) --
	(263.08,127.93) --
	(264.94,119.24) --
	(266.81,110.17) --
	(268.67,101.23) --
	(270.54, 92.80) --
	(272.40, 85.07) --
	(274.26, 78.36) --
	(276.13, 72.84) --
	(277.99, 68.63) --
	(279.86, 65.65) --
	(281.72, 63.93) --
	(283.58, 63.01) --
	(285.45, 62.76) --
	(287.31, 62.91) --
	(289.17, 63.29) --
	(291.04, 63.67) --
	(292.90, 64.02) --
	(294.77, 64.39) --
	(296.63, 64.69) --
	(298.49, 64.84) --
	(300.36, 65.23) --
	(302.22, 65.83) --
	(304.08, 66.95) --
	(305.95, 68.51) --
	(307.81, 70.30) --
	(309.68, 72.59) --
	(311.54, 75.03) --
	(313.40, 77.84) --
	(315.27, 81.06) --
	(317.13, 84.80) --
	(318.99, 88.85) --
	(320.86, 93.27) --
	(322.72, 97.89) --
	(324.59,102.85) --
	(326.45,108.01) --
	(328.31,113.46) --
	(330.18,119.01) --
	(332.04,124.58) --
	(333.90,130.13) --
	(335.77,135.34) --
	(337.63,140.00) --
	(339.50,144.01) --
	(341.36,147.27) --
	(343.22,149.64) --
	(345.09,151.01) --
	(346.95,151.19) --
	(348.82,150.13) --
	(350.68,147.62) --
	(352.54,143.82) --
	(354.41,138.56) --
	(356.27,132.10) --
	(358.13,124.79) --
	(360.00,116.84) --
	(361.86,108.48) --
	(363.73, 99.93) --
	(365.59, 91.76) --
	(367.45, 84.14) --
	(369.32, 77.38) --
	(371.18, 71.46) --
	(373.04, 66.73) --
	(374.91, 63.12) --
	(376.77, 60.54) --
	(378.64, 59.17) --
	(380.50, 58.77) --
	(382.36, 59.10) --
	(384.23, 60.01) --
	(386.09, 61.13) --
	(387.95, 62.48) --
	(389.82, 63.74) --
	(391.68, 65.00) --
	(393.55, 66.32) --
	(395.41, 67.59) --
	(397.27, 68.95) --
	(399.14, 70.31) --
	(401.00, 71.75) --
	(402.87, 73.57) --
	(404.73, 75.70) --
	(406.59, 78.18);
\definecolor{drawColor}{RGB}{51,26,204}

\path[draw=drawColor,draw opacity=0.76,line width= 1.2pt,line join=round,line cap=round] ( 61.79,186.85) --
	( 63.66,157.90) --
	( 65.52,137.11) --
	( 67.38,122.63) --
	( 69.25,112.52) --
	( 71.11,105.15) --
	( 72.97, 99.77) --
	( 74.84, 95.96) --
	( 76.70, 93.31) --
	( 78.57, 91.48) --
	( 80.43, 90.11) --
	( 82.29, 89.05) --
	( 84.16, 88.19) --
	( 86.02, 87.49) --
	( 87.89, 87.05) --
	( 89.75, 86.87) --
	( 91.61, 86.57) --
	( 93.48, 86.25) --
	( 95.34, 85.99) --
	( 97.20, 85.68) --
	( 99.07, 85.37) --
	(100.93, 85.03) --
	(102.80, 84.99) --
	(104.66, 84.87) --
	(106.52, 84.76) --
	(108.39, 84.66) --
	(110.25, 84.51) --
	(112.11, 84.44) --
	(113.98, 84.50) --
	(115.84, 84.84) --
	(117.71, 85.30) --
	(119.57, 85.73) --
	(121.43, 86.01) --
	(123.30, 86.44) --
	(125.16, 86.84) --
	(127.02, 87.21) --
	(128.89, 87.96) --
	(130.75, 88.84) --
	(132.62, 89.97) --
	(134.48, 91.34) --
	(136.34, 93.22) --
	(138.21, 95.53) --
	(140.07, 98.22) --
	(141.93,101.54) --
	(143.80,105.60) --
	(145.66,110.29) --
	(147.53,115.23) --
	(149.39,120.51) --
	(151.25,125.73) --
	(153.12,130.19) --
	(154.98,133.54) --
	(156.85,135.15) --
	(158.71,134.90) --
	(160.57,132.80) --
	(162.44,128.83) --
	(164.30,123.35) --
	(166.16,116.76) --
	(168.03,109.82) --
	(169.89,103.05) --
	(171.76, 96.81) --
	(173.62, 91.54) --
	(175.48, 87.69) --
	(177.35, 85.27) --
	(179.21, 83.94) --
	(181.07, 83.38) --
	(182.94, 83.61) --
	(184.80, 84.19) --
	(186.67, 84.99) --
	(188.53, 85.55) --
	(190.39, 85.88) --
	(192.26, 85.99) --
	(194.12, 85.82) --
	(195.98, 85.42) --
	(197.85, 85.05) --
	(199.71, 84.77) --
	(201.58, 84.59) --
	(203.44, 84.65) --
	(205.30, 84.86) --
	(207.17, 85.03) --
	(209.03, 85.31) --
	(210.89, 85.47) --
	(212.76, 85.65) --
	(214.62, 86.02) --
	(216.49, 86.36) --
	(218.35, 87.01) --
	(220.21, 87.89) --
	(222.08, 89.01) --
	(223.94, 90.10) --
	(225.81, 91.63) --
	(227.67, 93.24) --
	(229.53, 95.31) --
	(231.40, 97.63) --
	(233.26,100.22) --
	(235.12,103.19) --
	(236.99,106.48) --
	(238.85,110.03) --
	(240.72,113.53) --
	(242.58,117.08) --
	(244.44,120.37) --
	(246.31,123.26) --
	(248.17,125.55) --
	(250.03,126.96) --
	(251.90,127.53) --
	(253.76,126.98) --
	(255.63,125.32) --
	(257.49,122.66) --
	(259.35,119.06) --
	(261.22,114.66) --
	(263.08,109.76) --
	(264.94,104.69) --
	(266.81, 99.46) --
	(268.67, 94.75) --
	(270.54, 90.78) --
	(272.40, 87.55) --
	(274.26, 85.14) --
	(276.13, 83.41) --
	(277.99, 82.37) --
	(279.86, 81.89) --
	(281.72, 82.04) --
	(283.58, 82.61) --
	(285.45, 83.31) --
	(287.31, 84.13) --
	(289.17, 84.94) --
	(291.04, 85.39) --
	(292.90, 85.68) --
	(294.77, 85.71) --
	(296.63, 85.57) --
	(298.49, 85.63) --
	(300.36, 85.54) --
	(302.22, 85.43) --
	(304.08, 85.57) --
	(305.95, 85.94) --
	(307.81, 86.36) --
	(309.68, 87.00) --
	(311.54, 87.84) --
	(313.40, 88.90) --
	(315.27, 89.91) --
	(317.13, 91.12) --
	(318.99, 92.44) --
	(320.86, 94.05) --
	(322.72, 96.01) --
	(324.59, 98.35) --
	(326.45,100.84) --
	(328.31,103.53) --
	(330.18,106.41) --
	(332.04,109.28) --
	(333.90,112.19) --
	(335.77,115.08) --
	(337.63,117.64) --
	(339.50,119.88) --
	(341.36,121.54) --
	(343.22,122.74) --
	(345.09,123.36) --
	(346.95,123.13) --
	(348.82,122.08) --
	(350.68,120.40) --
	(352.54,118.03) --
	(354.41,114.99) --
	(356.27,111.44) --
	(358.13,107.51) --
	(360.00,103.37) --
	(361.86, 99.36) --
	(363.73, 95.48) --
	(365.59, 91.89) --
	(367.45, 88.87) --
	(369.32, 86.42) --
	(371.18, 84.49) --
	(373.04, 83.08) --
	(374.91, 82.14) --
	(376.77, 81.82) --
	(378.64, 81.82) --
	(380.50, 81.97) --
	(382.36, 82.52) --
	(384.23, 83.13) --
	(386.09, 83.89) --
	(387.95, 84.65) --
	(389.82, 85.28) --
	(391.68, 85.83) --
	(393.55, 86.22) --
	(395.41, 86.42) --
	(397.27, 86.73) --
	(399.14, 86.81) --
	(401.00, 87.15) --
	(402.87, 87.71) --
	(404.73, 88.60) --
	(406.59, 89.39);
\definecolor{drawColor}{RGB}{102,51,153}

\path[draw=drawColor,draw opacity=0.72,line width= 1.2pt,line join=round,line cap=round] ( 61.79,186.85) --
	( 63.66,148.60) --
	( 65.52,124.45) --
	( 67.38,110.78) --
	( 69.25,102.64) --
	( 71.11, 97.57) --
	( 72.97, 94.25) --
	( 74.84, 92.44) --
	( 76.70, 91.42) --
	( 78.57, 90.56) --
	( 80.43, 90.19) --
	( 82.29, 90.02) --
	( 84.16, 89.99) --
	( 86.02, 89.98) --
	( 87.89, 90.24) --
	( 89.75, 90.51) --
	( 91.61, 90.75) --
	( 93.48, 90.98) --
	( 95.34, 91.10) --
	( 97.20, 91.24) --
	( 99.07, 91.26) --
	(100.93, 91.05) --
	(102.80, 90.69) --
	(104.66, 90.42) --
	(106.52, 90.16) --
	(108.39, 90.09) --
	(110.25, 90.46) --
	(112.11, 90.36) --
	(113.98, 90.08) --
	(115.84, 90.55) --
	(117.71, 90.88) --
	(119.57, 91.22) --
	(121.43, 91.18) --
	(123.30, 91.24) --
	(125.16, 91.10) --
	(127.02, 90.93) --
	(128.89, 90.85) --
	(130.75, 91.39) --
	(132.62, 91.72) --
	(134.48, 92.17) --
	(136.34, 93.29) --
	(138.21, 94.88) --
	(140.07, 97.12) --
	(141.93, 99.80) --
	(143.80,102.62) --
	(145.66,106.20) --
	(147.53,109.84) --
	(149.39,113.28) --
	(151.25,116.89) --
	(153.12,119.71) --
	(154.98,121.68) --
	(156.85,122.72) --
	(158.71,122.02) --
	(160.57,119.92) --
	(162.44,116.83) --
	(164.30,112.66) --
	(166.16,107.62) --
	(168.03,102.73) --
	(169.89, 98.63) --
	(171.76, 94.84) --
	(173.62, 92.18) --
	(175.48, 90.14) --
	(177.35, 89.18) --
	(179.21, 88.63) --
	(181.07, 89.17) --
	(182.94, 89.45) --
	(184.80, 89.95) --
	(186.67, 90.50) --
	(188.53, 90.80) --
	(190.39, 90.85) --
	(192.26, 90.61) --
	(194.12, 90.44) --
	(195.98, 90.56) --
	(197.85, 90.61) --
	(199.71, 90.45) --
	(201.58, 90.58) --
	(203.44, 90.82) --
	(205.30, 90.88) --
	(207.17, 90.64) --
	(209.03, 90.52) --
	(210.89, 90.53) --
	(212.76, 90.68) --
	(214.62, 91.25) --
	(216.49, 91.74) --
	(218.35, 91.71) --
	(220.21, 92.09) --
	(222.08, 92.59) --
	(223.94, 93.47) --
	(225.81, 94.10) --
	(227.67, 95.18) --
	(229.53, 95.97) --
	(231.40, 96.80) --
	(233.26, 98.35) --
	(235.12,100.41) --
	(236.99,102.63) --
	(238.85,105.33) --
	(240.72,108.04) --
	(242.58,110.63) --
	(244.44,112.62) --
	(246.31,114.18) --
	(248.17,115.31) --
	(250.03,115.74) --
	(251.90,115.52) --
	(253.76,114.77) --
	(255.63,113.56) --
	(257.49,111.38) --
	(259.35,108.73) --
	(261.22,106.12) --
	(263.08,103.45) --
	(264.94,100.72) --
	(266.81, 97.66) --
	(268.67, 94.88) --
	(270.54, 92.69) --
	(272.40, 91.35) --
	(274.26, 90.17) --
	(276.13, 89.50) --
	(277.99, 89.04) --
	(279.86, 89.03) --
	(281.72, 88.93) --
	(283.58, 89.31) --
	(285.45, 89.74) --
	(287.31, 90.33) --
	(289.17, 90.80) --
	(291.04, 90.57) --
	(292.90, 90.59) --
	(294.77, 90.67) --
	(296.63, 90.96) --
	(298.49, 91.11) --
	(300.36, 90.89) --
	(302.22, 90.87) --
	(304.08, 91.04) --
	(305.95, 91.17) --
	(307.81, 91.38) --
	(309.68, 91.81) --
	(311.54, 92.44) --
	(313.40, 92.97) --
	(315.27, 93.85) --
	(317.13, 94.64) --
	(318.99, 95.46) --
	(320.86, 96.30) --
	(322.72, 97.42) --
	(324.59, 98.17) --
	(326.45, 99.60) --
	(328.31,101.34) --
	(330.18,103.16) --
	(332.04,104.86) --
	(333.90,106.46) --
	(335.77,107.90) --
	(337.63,109.21) --
	(339.50,110.30) --
	(341.36,111.42) --
	(343.22,111.97) --
	(345.09,111.92) --
	(346.95,111.67) --
	(348.82,111.04) --
	(350.68,110.18) --
	(352.54,108.91) --
	(354.41,107.00) --
	(356.27,104.57) --
	(358.13,102.16) --
	(360.00, 99.65) --
	(361.86, 97.51) --
	(363.73, 95.08) --
	(365.59, 93.27) --
	(367.45, 92.22) --
	(369.32, 91.35) --
	(371.18, 90.53) --
	(373.04, 90.15) --
	(374.91, 89.85) --
	(376.77, 89.58) --
	(378.64, 89.44) --
	(380.50, 89.52) --
	(382.36, 89.94) --
	(384.23, 90.13) --
	(386.09, 90.08) --
	(387.95, 90.14) --
	(389.82, 90.40) --
	(391.68, 90.63) --
	(393.55, 91.07) --
	(395.41, 91.28) --
	(397.27, 91.46) --
	(399.14, 91.61) --
	(401.00, 92.06) --
	(402.87, 92.34) --
	(404.73, 93.08) --
	(406.59, 93.98);
\definecolor{drawColor}{RGB}{153,77,102}

\path[draw=drawColor,draw opacity=0.68,line width= 1.2pt,line join=round,line cap=round] ( 61.79,186.85) --
	( 63.66,138.69) --
	( 65.52,114.05) --
	( 67.38,103.86) --
	( 69.25, 99.02) --
	( 71.11, 96.52) --
	( 72.97, 94.51) --
	( 74.84, 93.46) --
	( 76.70, 92.88) --
	( 78.57, 92.47) --
	( 80.43, 92.70) --
	( 82.29, 92.33) --
	( 84.16, 92.58) --
	( 86.02, 92.71) --
	( 87.89, 92.60) --
	( 89.75, 92.52) --
	( 91.61, 92.56) --
	( 93.48, 92.12) --
	( 95.34, 91.89) --
	( 97.20, 92.03) --
	( 99.07, 92.26) --
	(100.93, 91.98) --
	(102.80, 91.99) --
	(104.66, 92.18) --
	(106.52, 92.35) --
	(108.39, 92.63) --
	(110.25, 92.51) --
	(112.11, 93.00) --
	(113.98, 92.87) --
	(115.84, 92.46) --
	(117.71, 91.89) --
	(119.57, 91.92) --
	(121.43, 92.21) --
	(123.30, 92.67) --
	(125.16, 93.15) --
	(127.02, 92.89) --
	(128.89, 92.55) --
	(130.75, 93.01) --
	(132.62, 93.34) --
	(134.48, 94.28) --
	(136.34, 95.66) --
	(138.21, 97.11) --
	(140.07, 98.90) --
	(141.93,100.62) --
	(143.80,103.02) --
	(145.66,105.52) --
	(147.53,108.07) --
	(149.39,110.25) --
	(151.25,112.03) --
	(153.12,113.32) --
	(154.98,114.06) --
	(156.85,113.93) --
	(158.71,112.83) --
	(160.57,111.01) --
	(162.44,109.00) --
	(164.30,106.08) --
	(166.16,102.99) --
	(168.03,100.36) --
	(169.89, 98.42) --
	(171.76, 96.73) --
	(173.62, 95.33) --
	(175.48, 94.42) --
	(177.35, 93.71) --
	(179.21, 93.10) --
	(181.07, 92.71) --
	(182.94, 92.41) --
	(184.80, 92.41) --
	(186.67, 92.32) --
	(188.53, 91.96) --
	(190.39, 92.12) --
	(192.26, 91.94) --
	(194.12, 91.87) --
	(195.98, 92.07) --
	(197.85, 92.70) --
	(199.71, 92.81) --
	(201.58, 92.79) --
	(203.44, 92.49) --
	(205.30, 92.53) --
	(207.17, 92.49) --
	(209.03, 92.41) --
	(210.89, 92.19) --
	(212.76, 92.82) --
	(214.62, 92.88) --
	(216.49, 92.87) --
	(218.35, 93.39) --
	(220.21, 93.98) --
	(222.08, 94.55) --
	(223.94, 94.92) --
	(225.81, 95.35) --
	(227.67, 96.37) --
	(229.53, 97.57) --
	(231.40, 98.68) --
	(233.26, 99.99) --
	(235.12,101.62) --
	(236.99,103.16) --
	(238.85,104.94) --
	(240.72,106.47) --
	(242.58,107.43) --
	(244.44,108.20) --
	(246.31,108.35) --
	(248.17,108.86) --
	(250.03,108.56) --
	(251.90,108.07) --
	(253.76,107.20) --
	(255.63,106.49) --
	(257.49,105.54) --
	(259.35,104.15) --
	(261.22,102.33) --
	(263.08,100.40) --
	(264.94, 98.54) --
	(266.81, 97.38) --
	(268.67, 96.13) --
	(270.54, 94.94) --
	(272.40, 94.14) --
	(274.26, 93.77) --
	(276.13, 93.74) --
	(277.99, 93.32) --
	(279.86, 92.52) --
	(281.72, 92.36) --
	(283.58, 92.25) --
	(285.45, 92.17) --
	(287.31, 91.94) --
	(289.17, 92.26) --
	(291.04, 92.36) --
	(292.90, 92.75) --
	(294.77, 93.03) --
	(296.63, 93.03) --
	(298.49, 93.14) --
	(300.36, 93.43) --
	(302.22, 93.17) --
	(304.08, 93.23) --
	(305.95, 93.09) --
	(307.81, 92.94) --
	(309.68, 93.34) --
	(311.54, 94.17) --
	(313.40, 94.77) --
	(315.27, 95.28) --
	(317.13, 96.12) --
	(318.99, 96.97) --
	(320.86, 97.76) --
	(322.72, 98.51) --
	(324.59, 99.34) --
	(326.45,100.48) --
	(328.31,101.58) --
	(330.18,102.47) --
	(332.04,103.75) --
	(333.90,104.67) --
	(335.77,105.32) --
	(337.63,105.74) --
	(339.50,106.18) --
	(341.36,106.48) --
	(343.22,106.43) --
	(345.09,105.89) --
	(346.95,105.38) --
	(348.82,104.71) --
	(350.68,103.67) --
	(352.54,102.76) --
	(354.41,101.67) --
	(356.27,101.02) --
	(358.13, 99.76) --
	(360.00, 98.62) --
	(361.86, 97.35) --
	(363.73, 95.99) --
	(365.59, 94.86) --
	(367.45, 94.52) --
	(369.32, 94.25) --
	(371.18, 93.82) --
	(373.04, 93.46) --
	(374.91, 93.72) --
	(376.77, 93.47) --
	(378.64, 93.26) --
	(380.50, 93.11) --
	(382.36, 92.88) --
	(384.23, 92.96) --
	(386.09, 92.66) --
	(387.95, 92.73) --
	(389.82, 92.96) --
	(391.68, 92.89) --
	(393.55, 92.98) --
	(395.41, 93.15) --
	(397.27, 93.31) --
	(399.14, 93.68) --
	(401.00, 94.34) --
	(402.87, 94.70) --
	(404.73, 94.86) --
	(406.59, 95.01);
\definecolor{drawColor}{RGB}{204,102,51}

\path[draw=drawColor,draw opacity=0.64,line width= 1.2pt,line join=round,line cap=round] ( 61.79,186.85) --
	( 63.66,120.87) --
	( 65.52, 99.59) --
	( 67.38, 97.00) --
	( 69.25, 96.19) --
	( 71.11, 95.17) --
	( 72.97, 95.27) --
	( 74.84, 94.81) --
	( 76.70, 94.98) --
	( 78.57, 94.95) --
	( 80.43, 94.98) --
	( 82.29, 94.93) --
	( 84.16, 94.89) --
	( 86.02, 94.90) --
	( 87.89, 94.75) --
	( 89.75, 94.81) --
	( 91.61, 94.82) --
	( 93.48, 94.70) --
	( 95.34, 94.81) --
	( 97.20, 95.04) --
	( 99.07, 94.75) --
	(100.93, 94.86) --
	(102.80, 94.76) --
	(104.66, 95.05) --
	(106.52, 95.53) --
	(108.39, 95.25) --
	(110.25, 95.10) --
	(112.11, 94.72) --
	(113.98, 95.07) --
	(115.84, 94.76) --
	(117.71, 94.74) --
	(119.57, 94.59) --
	(121.43, 94.66) --
	(123.30, 95.04) --
	(125.16, 95.31) --
	(127.02, 95.49) --
	(128.89, 95.38) --
	(130.75, 95.68) --
	(132.62, 96.72) --
	(134.48, 97.62) --
	(136.34, 98.59) --
	(138.21, 99.41) --
	(140.07,100.44) --
	(141.93,101.29) --
	(143.80,102.28) --
	(145.66,103.08) --
	(147.53,103.68) --
	(149.39,103.48) --
	(151.25,103.18) --
	(153.12,103.55) --
	(154.98,103.12) --
	(156.85,102.97) --
	(158.71,102.07) --
	(160.57,101.31) --
	(162.44,100.50) --
	(164.30,100.02) --
	(166.16, 99.09) --
	(168.03, 98.87) --
	(169.89, 98.43) --
	(171.76, 97.90) --
	(173.62, 97.40) --
	(175.48, 96.86) --
	(177.35, 96.05) --
	(179.21, 96.06) --
	(181.07, 96.36) --
	(182.94, 96.09) --
	(184.80, 96.04) --
	(186.67, 96.48) --
	(188.53, 96.47) --
	(190.39, 95.50) --
	(192.26, 94.63) --
	(194.12, 94.91) --
	(195.98, 95.18) --
	(197.85, 95.70) --
	(199.71, 95.51) --
	(201.58, 95.10) --
	(203.44, 95.03) --
	(205.30, 94.90) --
	(207.17, 94.87) --
	(209.03, 94.96) --
	(210.89, 95.53) --
	(212.76, 96.07) --
	(214.62, 95.61) --
	(216.49, 96.13) --
	(218.35, 96.88) --
	(220.21, 96.93) --
	(222.08, 97.34) --
	(223.94, 97.88) --
	(225.81, 98.23) --
	(227.67, 98.86) --
	(229.53, 99.79) --
	(231.40, 99.87) --
	(233.26,100.02) --
	(235.12,100.70) --
	(236.99,100.48) --
	(238.85,100.85) --
	(240.72,100.72) --
	(242.58,100.37) --
	(244.44,100.87) --
	(246.31,100.86) --
	(248.17,100.62) --
	(250.03,100.19) --
	(251.90,100.40) --
	(253.76,100.23) --
	(255.63, 99.84) --
	(257.49, 99.24) --
	(259.35, 98.92) --
	(261.22, 98.48) --
	(263.08, 98.44) --
	(264.94, 97.71) --
	(266.81, 97.37) --
	(268.67, 97.45) --
	(270.54, 97.25) --
	(272.40, 96.87) --
	(274.26, 96.58) --
	(276.13, 96.70) --
	(277.99, 96.59) --
	(279.86, 96.36) --
	(281.72, 96.01) --
	(283.58, 96.13) --
	(285.45, 96.06) --
	(287.31, 96.02) --
	(289.17, 95.79) --
	(291.04, 95.71) --
	(292.90, 95.90) --
	(294.77, 95.69) --
	(296.63, 95.99) --
	(298.49, 95.70) --
	(300.36, 95.69) --
	(302.22, 95.79) --
	(304.08, 96.26) --
	(305.95, 96.39) --
	(307.81, 96.95) --
	(309.68, 97.66) --
	(311.54, 97.82) --
	(313.40, 98.12) --
	(315.27, 98.38) --
	(317.13, 98.44) --
	(318.99, 98.88) --
	(320.86, 98.97) --
	(322.72, 99.16) --
	(324.59, 99.65) --
	(326.45, 99.73) --
	(328.31, 99.70) --
	(330.18, 99.27) --
	(332.04, 99.38) --
	(333.90, 99.51) --
	(335.77, 99.55) --
	(337.63, 99.77) --
	(339.50, 99.57) --
	(341.36, 99.10) --
	(343.22, 99.16) --
	(345.09, 98.88) --
	(346.95, 98.93) --
	(348.82, 99.01) --
	(350.68, 98.71) --
	(352.54, 98.30) --
	(354.41, 98.84) --
	(356.27, 98.60) --
	(358.13, 97.51) --
	(360.00, 97.41) --
	(361.86, 97.79) --
	(363.73, 97.22) --
	(365.59, 97.20) --
	(367.45, 97.30) --
	(369.32, 97.13) --
	(371.18, 96.91) --
	(373.04, 96.61) --
	(374.91, 96.31) --
	(376.77, 96.53) --
	(378.64, 96.17) --
	(380.50, 96.30) --
	(382.36, 96.23) --
	(384.23, 96.31) --
	(386.09, 96.25) --
	(387.95, 96.27) --
	(389.82, 96.75) --
	(391.68, 96.81) --
	(393.55, 96.74) --
	(395.41, 96.95) --
	(397.27, 97.24) --
	(399.14, 97.35) --
	(401.00, 97.69) --
	(402.87, 97.79) --
	(404.73, 97.70) --
	(406.59, 97.50);
\definecolor{drawColor}{RGB}{255,128,0}

\path[draw=drawColor,draw opacity=0.60,line width= 1.2pt,line join=round,line cap=round] ( 61.79,186.85) --
	( 63.66,105.83) --
	( 65.52, 93.82) --
	( 67.38, 95.80) --
	( 69.25, 95.51) --
	( 71.11, 96.01) --
	( 72.97, 96.14) --
	( 74.84, 96.37) --
	( 76.70, 96.15) --
	( 78.57, 96.37) --
	( 80.43, 96.66) --
	( 82.29, 96.00) --
	( 84.16, 96.35) --
	( 86.02, 96.13) --
	( 87.89, 96.03) --
	( 89.75, 96.10) --
	( 91.61, 96.32) --
	( 93.48, 96.49) --
	( 95.34, 95.99) --
	( 97.20, 95.81) --
	( 99.07, 96.18) --
	(100.93, 96.07) --
	(102.80, 95.98) --
	(104.66, 96.14) --
	(106.52, 95.75) --
	(108.39, 95.32) --
	(110.25, 95.77) --
	(112.11, 96.01) --
	(113.98, 96.01) --
	(115.84, 95.89) --
	(117.71, 96.02) --
	(119.57, 96.87) --
	(121.43, 96.58) --
	(123.30, 96.38) --
	(125.16, 96.91) --
	(127.02, 97.28) --
	(128.89, 97.32) --
	(130.75, 98.03) --
	(132.62, 97.80) --
	(134.48, 98.64) --
	(136.34, 98.84) --
	(138.21, 99.46) --
	(140.07, 99.96) --
	(141.93, 99.70) --
	(143.80, 99.94) --
	(145.66, 99.99) --
	(147.53, 99.74) --
	(149.39, 99.68) --
	(151.25, 99.72) --
	(153.12, 99.66) --
	(154.98, 99.91) --
	(156.85, 99.90) --
	(158.71, 99.62) --
	(160.57, 99.74) --
	(162.44, 99.16) --
	(164.30, 98.16) --
	(166.16, 98.20) --
	(168.03, 98.56) --
	(169.89, 97.82) --
	(171.76, 98.23) --
	(173.62, 98.15) --
	(175.48, 97.88) --
	(177.35, 97.31) --
	(179.21, 97.29) --
	(181.07, 96.92) --
	(182.94, 97.01) --
	(184.80, 96.97) --
	(186.67, 96.98) --
	(188.53, 97.10) --
	(190.39, 96.40) --
	(192.26, 96.21) --
	(194.12, 95.96) --
	(195.98, 96.78) --
	(197.85, 96.58) --
	(199.71, 96.58) --
	(201.58, 96.17) --
	(203.44, 96.61) --
	(205.30, 96.91) --
	(207.17, 96.58) --
	(209.03, 97.16) --
	(210.89, 97.58) --
	(212.76, 97.49) --
	(214.62, 97.26) --
	(216.49, 96.82) --
	(218.35, 97.64) --
	(220.21, 98.35) --
	(222.08, 98.21) --
	(223.94, 98.02) --
	(225.81, 97.91) --
	(227.67, 98.35) --
	(229.53, 98.27) --
	(231.40, 99.02) --
	(233.26, 98.57) --
	(235.12, 98.59) --
	(236.99, 98.03) --
	(238.85, 98.74) --
	(240.72, 98.83) --
	(242.58, 98.70) --
	(244.44, 98.61) --
	(246.31, 98.98) --
	(248.17, 98.99) --
	(250.03, 98.56) --
	(251.90, 98.62) --
	(253.76, 98.51) --
	(255.63, 97.83) --
	(257.49, 98.53) --
	(259.35, 98.74) --
	(261.22, 98.56) --
	(263.08, 97.99) --
	(264.94, 97.96) --
	(266.81, 97.80) --
	(268.67, 97.75) --
	(270.54, 97.43) --
	(272.40, 97.95) --
	(274.26, 98.18) --
	(276.13, 97.27) --
	(277.99, 96.93) --
	(279.86, 96.99) --
	(281.72, 97.37) --
	(283.58, 96.71) --
	(285.45, 96.96) --
	(287.31, 97.37) --
	(289.17, 96.83) --
	(291.04, 97.19) --
	(292.90, 97.57) --
	(294.77, 97.36) --
	(296.63, 97.64) --
	(298.49, 96.92) --
	(300.36, 96.77) --
	(302.22, 97.45) --
	(304.08, 97.66) --
	(305.95, 97.63) --
	(307.81, 97.38) --
	(309.68, 97.98) --
	(311.54, 97.72) --
	(313.40, 98.11) --
	(315.27, 97.74) --
	(317.13, 98.45) --
	(318.99, 97.68) --
	(320.86, 98.17) --
	(322.72, 97.98) --
	(324.59, 97.70) --
	(326.45, 98.12) --
	(328.31, 98.59) --
	(330.18, 98.54) --
	(332.04, 97.60) --
	(333.90, 98.72) --
	(335.77, 98.86) --
	(337.63, 97.71) --
	(339.50, 98.04) --
	(341.36, 98.50) --
	(343.22, 98.32) --
	(345.09, 98.63) --
	(346.95, 98.38) --
	(348.82, 98.19) --
	(350.68, 98.59) --
	(352.54, 97.79) --
	(354.41, 97.95) --
	(356.27, 98.09) --
	(358.13, 97.48) --
	(360.00, 97.69) --
	(361.86, 97.59) --
	(363.73, 98.02) --
	(365.59, 97.49) --
	(367.45, 97.64) --
	(369.32, 97.48) --
	(371.18, 97.35) --
	(373.04, 97.86) --
	(374.91, 97.68) --
	(376.77, 97.69) --
	(378.64, 97.52) --
	(380.50, 97.31) --
	(382.36, 98.04) --
	(384.23, 97.46) --
	(386.09, 97.88) --
	(387.95, 97.90) --
	(389.82, 97.71) --
	(391.68, 97.21) --
	(393.55, 97.42) --
	(395.41, 96.98) --
	(397.27, 97.02) --
	(399.14, 97.78) --
	(401.00, 97.82) --
	(402.87, 97.19) --
	(404.73, 98.14) --
	(406.59, 97.64);
\definecolor{fillColor}{RGB}{255,255,255}

\path[fill=fillColor,fill opacity=0.93] (368.66,115.65) --
	(368.66,186.85) --
	(406.59,186.85) --
	(406.59,115.65) --
	(406.59,115.65) --
	cycle;
\definecolor{drawColor}{RGB}{0,0,0}

\path[draw=drawColor,line width= 0.4pt,line join=round,line cap=round] (368.43,192.19) rectangle (420.38,115.39);
\definecolor{drawColor}{RGB}{0,0,255}

\path[draw=drawColor,draw opacity=0.80,line width= 1.2pt,line join=round,line cap=round] (375.63,172.99) -- (390.03,172.99);
\definecolor{drawColor}{RGB}{51,26,204}

\path[draw=drawColor,draw opacity=0.76,line width= 1.2pt,line join=round,line cap=round] (375.63,163.39) -- (390.03,163.39);
\definecolor{drawColor}{RGB}{102,51,153}

\path[draw=drawColor,draw opacity=0.72,line width= 1.2pt,line join=round,line cap=round] (375.63,153.79) -- (390.03,153.79);
\definecolor{drawColor}{RGB}{153,77,102}

\path[draw=drawColor,draw opacity=0.68,line width= 1.2pt,line join=round,line cap=round] (375.63,144.19) -- (390.03,144.19);
\definecolor{drawColor}{RGB}{204,102,51}

\path[draw=drawColor,draw opacity=0.64,line width= 1.2pt,line join=round,line cap=round] (375.63,134.59) -- (390.03,134.59);
\definecolor{drawColor}{RGB}{255,128,0}

\path[draw=drawColor,draw opacity=0.60,line width= 1.2pt,line join=round,line cap=round] (375.63,124.99) -- (390.03,124.99);
\definecolor{drawColor}{RGB}{0,0,0}

\node[text=drawColor,anchor=base,inner sep=0pt, outer sep=0pt, scale=  0.70] at (394.41,182.59) {$\alpha$ [s$^{-2}$]};

\node[text=drawColor,anchor=base west,inner sep=0pt, outer sep=0pt, scale=  0.70] at (397.23,170.24) {$0$};

\node[text=drawColor,anchor=base west,inner sep=0pt, outer sep=0pt, scale=  0.70] at (397.23,160.64) {$0.05~~$};

\node[text=drawColor,anchor=base west,inner sep=0pt, outer sep=0pt, scale=  0.70] at (397.23,151.04) {$0.1$};

\node[text=drawColor,anchor=base west,inner sep=0pt, outer sep=0pt, scale=  0.70] at (397.23,141.44) {$0.2$};

\node[text=drawColor,anchor=base west,inner sep=0pt, outer sep=0pt, scale=  0.70] at (397.23,131.84) {$0.5$};

\node[text=drawColor,anchor=base west,inner sep=0pt, outer sep=0pt, scale=  0.70] at (397.23,122.24) {$1$};
\end{scope}
\end{tikzpicture}

%% file: Article-PHS_Traffic_REV1.bbl
\begin{thebibliography}{10}

\bibitem{Pipes1953}
Louis~A. Pipes.
\newblock An operational analysis of traffic dynamics.
\newblock {\em Journal of Applied Physics}, 24(3):274--281, 1953.

\bibitem{Chandler1958}
Robert~E. Chandler, Robert Herman, and Elliott~W. Montroll.
\newblock Traffic dynamics: Studies in car following.
\newblock {\em Operational Research}, 6(2):165--184, 1958.

\bibitem{Herman1959}
Robert Herman, Elliott~W. Montroll, Renfrey~B. Potts, and Richard~W. Rothery.
\newblock Traffic dynamics: analysis of stability in car following.
\newblock {\em Operational Research}, 7(1):86--106, 1959.

\bibitem{gazis1961nonlinear}
Denos~C Gazis, Robert Herman, and Richard~W Rothery.
\newblock Nonlinear follow-the-leader models of traffic flow.
\newblock {\em Operational Research}, 9(4):545--567, 1961.

\bibitem{bando1995dynamical}
Masako Bando, Katsuya Hasebe, Akihiro Nakayama, Akihiro Shibata, and Yuki
  Sugiyama.
\newblock Dynamical model of traffic congestion and numerical simulation.
\newblock {\em Physical Review E}, 51(2):1035, 1995.

\bibitem{Orosz2009}
G{\'a}bor Orosz, R.~Eddie Wilson, R{\'o}bert Szalai, and G{\'a}bor
  St{\'e}p{\'a}n.
\newblock Exciting traffic jams: Nonlinear phenomena behind traffic jam
  formation on highways.
\newblock {\em Physical Review E}, 80(4):046205, 2009.

\bibitem{Orosz2010}
G{\'a}bor Orosz, R.~Eddie Wilson, and G{\'a}bor St{\'e}p{\'a}n.
\newblock Traffic jams : dynamics and control.
\newblock {\em Proceedings of the Royal Society A}, 368(1957):4455--4479, 2010.

\bibitem{NagelS92A}
Kai Nagel and Michael Schreckenberg.
\newblock A cellular automaton model for freeway traffic.
\newblock {\em Journal de Physique I}, 2:2221, 1992.

\bibitem{barlovic1998metastable}
Robert Barlovic, Ludger Santen, Andreas Schadschneider, and Michael
  Schreckenberg.
\newblock Metastable states in cellular automata for traffic flow.
\newblock {\em European Physical Journal B}, 5(3):793--800, 1998.

\bibitem{treiber2009hamilton}
Martin Treiber and Dirk Helbing.
\newblock Hamilton-like statistics in onedimensional driven dissipative
  many-particle systems.
\newblock {\em The European Physical Journal B}, 68:607--618, 2009.

\bibitem{Hamdar2015}
Samer~H. Hamdar, Hani~S. Mahmassani, and Martin Treiber.
\newblock From behavioral psychology to acceleration modeling: Calibration,
  validation, and exploration of drivers’ cognitive and safety parameters in
  a risk-taking environment.
\newblock {\em Transportation Research Part B: Methodological}, 78:32--53,
  2015.

\bibitem{tordeux2016white}
Antoine Tordeux and Andreas Schadschneider.
\newblock White and relaxed noises in optimal velocity models for pedestrian
  flow with stop-and-go waves.
\newblock {\em Journal of Physics A: Mathematical and Theoretical},
  49(18):185101, 2016.

\bibitem{treiber2017intelligent}
Martin Treiber and Arne Kesting.
\newblock The intelligent driver model with stochasticity-new insights into
  traffic flow oscillations.
\newblock {\em Transportation Research Procedia}, 23:174--187, 2017.

\bibitem{wang2020stability}
Yu~Wang, Xiaopeng Li, Junfang Tian, and Rui Jiang.
\newblock Stability analysis of stochastic linear car-following models.
\newblock {\em Transportation Science}, 54(1):274--297, 2020.

\bibitem{friesen2021spontaneous}
Martin Friesen, Hanno Gottschalk, Barbara R{\"u}diger, and Antoine Tordeux.
\newblock Spontaneous wave formation in stochastic self-driven particle
  systems.
\newblock {\em SIAM Journal on Applied Mathematics}, 81(3):853--870, 2021.

\bibitem{ngoduy2019langevin}
Dong Ngoduy, Seunghyeon Lee, Martin Treiber, Mehdi Keyvan-Ekbatani, and Hai~L
  Vu.
\newblock Langevin method for a continuous stochastic car-following model and
  its stability conditions.
\newblock {\em Transportation Research Part C: Emerging Technologies},
  105:599--610, 2019.

\bibitem{xu2019analysis}
Tu~Xu and Jorge~A Laval.
\newblock Analysis of a two-regime stochastic car-following model: Explaining
  capacity drop and oscillation instabilities.
\newblock {\em Transportation Research Record}, 2673(10):610--619, 2019.

\bibitem{Stern2018}
Raphael~E. Stern, Shumo Cui, Maria~Laura {Delle Monache}, Rahul Bhadani, Matt
  Bunting, Miles Churchill, Nathaniel Hamilton, R’mani Haulcy, Hannah
  Pohlmann, Fangyu Wu, Benedetto Piccoli, Benjamin Seibold, Jonathan Sprinkle,
  and Daniel~B. Work.
\newblock Dissipation of stop-and-go waves via control of autonomous vehicles:
  Field experiments.
\newblock {\em Transportation Research Part C: Emerging Technologies},
  89:205--221, 2018.

\bibitem{gunter2020commercially}
George Gunter, Derek Gloudemans, Raphael~E Stern, Sean McQuade, Rahul Bhadani,
  Matt Bunting, Maria~Laura Delle~Monache, Roman Lysecky, Benjamin Seibold,
  Jonathan Sprinkle, et~al.
\newblock Are commercially implemented adaptive cruise control systems string
  stable?
\newblock {\em IEEE Transactions on Intelligent Transportation Systems},
  22(11):6992--7003, 2020.

\bibitem{makridis2021openacc}
Michail Makridis, Konstantinos Mattas, Aikaterini Anesiadou, and Biagio Ciuffo.
\newblock Openacc. an open database of car-following experiments to study the
  properties of commercial acc systems.
\newblock {\em Transportation Research Part C: Emerging Technologies},
  125:103047, 2021.

\bibitem{CIUFFO2021}
Biagio Ciuffo, Konstantinos Mattas, Michail Makridis, Giovanni Albano,
  Aikaterini Anesiadou, Yinglong He, Szilárd Josvai, Dimitris Komnos, Marton
  Pataki, Sandor Vass, and Zsolt Szalay.
\newblock Requiem on the positive effects of commercial adaptive cruise control
  on motorway traffic and recommendations for future automated driving systems.
\newblock {\em Transportation Research Part C: Emerging Technologies},
  130:103305, 2021.

\bibitem{treiber2006delays}
Martin Treiber, Arne Kesting, and Dirk Helbing.
\newblock Delays, inaccuracies and anticipation in microscopic traffic models.
\newblock {\em Physica A: Statistical Mechanics and its Applications},
  360(1):71--88, 2006.

\bibitem{wang2019effect}
Tao Wang, Guangyao Li, Jing Zhang, Shubin Li, and Tao Sun.
\newblock The effect of headway variation tendency on traffic flow: Modeling
  and stabilization.
\newblock {\em Physica A: Statistical Mechanics and Its Applications},
  525:566--575, 2019.

\bibitem{khound2021extending}
Parthib Khound, Peter Will, Antoine Tordeux, and Frank Gronwald.
\newblock Extending the adaptive time gap car-following model to enhance local
  and string stability for adaptive cruise control systems.
\newblock {\em Journal of Intelligent Transportation Systems}, pages 1--21,
  2021.

\bibitem{van2006port}
Arjan van~der Schaft.
\newblock Port-{H}amiltonian systems: an introductory survey.
\newblock In {\em Proceedings of the international congress of mathematicians},
  volume~3, pages 1339--1365. Citeseer, 2006.

\bibitem{van2014port}
Arjan van~der Schaft and Dimitri Jeltsema.
\newblock Port-{H}amiltonian systems theory: {A}n introductory overview.
\newblock {\em Foundations and Trends in Systems and Control}, 1(2-3):173--378,
  2014.

\bibitem{van1981symmetries}
Arjan van~der Schaft.
\newblock Symmetries and conservation laws for {H}amiltonian systems with
  inputs and outputs: A generalization of {Noether}'s theorem.
\newblock {\em Systems \& Control Letters}, 1(2), 1981.

\bibitem{MaschkeVdsBreedveld}
Bernhard Maschke, Arjan Van~Der Schaft, and Pieter~Cornelis Breedveld.
\newblock An intrinsic {H}amiltonian formulation of network dynamics:
  Non-standard poisson structures and gyrators.
\newblock {\em Journal of the Franklin Institute}, 329(5):923--966, 1992.

\bibitem{rashad2020twenty}
Ramy Rashad, Federico Califano, Arjan van~der Schaft, and Stefano Stramigioli.
\newblock Twenty years of distributed port-{H}amiltonian systems: a literature
  review.
\newblock {\em IMA Journal of Mathematical Control and Information},
  37(4):1400--1422, 2020.

\bibitem{knorn2015overview}
Steffi Knorn, Zhiyong Chen, and Richard~H Middleton.
\newblock Overview: Collective control of multiagent systems.
\newblock {\em IEEE Transactions on Control of Network Systems}, 3(4):334--347,
  2015.

\bibitem{wang2016output}
Bing Wang, Xinghu Wang, and Honghua Wang.
\newblock Output synchronization of multi-agent port-{H}amiltonian systems with
  link dynamics.
\newblock {\em Kybernetika}, 52(1):89--105, 2016.

\bibitem{cristofaro2022fault}
Andrea Cristofaro, Gaetano Giunta, and Paolo~Robuffo Giordano.
\newblock Fault-tolerant formation control of passive multi-agent systems using
  energy tanks.
\newblock {\em IEEE Control Systems Letters}, 2022.

\bibitem{van2010Consensus}
Arjan van~der Schaft and Bernhard Maschke.
\newblock Port-{H}amiltonian dynamics on graphs: {C}onsensus and coordination
  control algorithms.
\newblock {\em IFAC Proceedings Volumes}, 43(19):175--178, 2010.

\bibitem{chang2014protocol}
Li~Chang-Sheng and Wang Yu-Zhen.
\newblock Protocol design for output consensus of port-controlled {H}amiltonian
  multi-agent systems.
\newblock {\em Acta Automatica Sinica}, 40(3):415--422, 2014.

\bibitem{jafarian2015formation}
Matin Jafarian, Ewoud Vos, Claudio De~Persis, Arjan van~der Schaft, and
  Jacquelien~MA Scherpen.
\newblock Formation control of a multi-agent system subject to coulomb
  friction.
\newblock {\em Automatica}, 61:253--262, 2015.

\bibitem{WEI2017Consensus}
Jieqiang Wei, Anneroos~R.F. Everts, M.~Kanat Camlibel, and Arjan van~der
  Schaft.
\newblock Consensus dynamics with arbitrary sign-preserving nonlinearities.
\newblock {\em Automatica}, 83:226--233, 2017.

\bibitem{xue2019opinion}
Dong Xue, Sandra Hirche, and Ming Cao.
\newblock Opinion behavior analysis in social networks under the influence of
  coopetitive media.
\newblock {\em IEEE Transactions on Network Science and Engineering},
  7(3):961--974, 2019.

\bibitem{sharf2019analysis}
Miel Sharf and Daniel Zelazo.
\newblock Analysis and synthesis of mimo multi-agent systems using network
  optimization.
\newblock {\em IEEE Transactions on Automatic Control}, 64(11):4512--4524,
  2019.

\bibitem{matei2019inferring}
Ion Matei, Christos Mavridis, John~S. Baras, and Maksym Zhenirovskyy.
\newblock Inferring particle interaction physical models and their dynamical
  properties.
\newblock In {\em 2019 IEEE 58th Conference on Decision and Control (CDC)},
  pages 4615--4621. IEEE, 2019.

\bibitem{mavridis2020detection}
Christos~N. Mavridis, Nilesh Suriyarachchi, and John~S. Baras.
\newblock Detection of dynamically changing leaders in complex swarms from
  observed dynamic data.
\newblock In {\em International Conference on Decision and Game Theory for
  Security}, pages 223--240. Springer, 2020.

\bibitem{ma2021path}
Yan Ma, Jian Chen, Junmin Wang, Yanchuan Xu, and Yuexuan Wang.
\newblock Path-tracking considering yaw stability with passivity-based control
  for autonomous vehicles.
\newblock {\em IEEE Transactions on Intelligent Transportation Systems},
  23(7):8736--8746, 2021.

\bibitem{knorn2014passivity}
Steffi Knorn, Alejandro Donaire, Juan~C. Ag{\"u}ero, and Richard~H. Middleton.
\newblock Passivity-based control for multi-vehicle systems subject to string
  constraints.
\newblock {\em Automatica}, 50(12):3224--3230, 2014.

\bibitem{knorn2014scalability}
Steffi Knorn, Alejandro Donaire, Juan~C. Ag{\"u}ero, and Richard~H. Middleton.
\newblock Scalability of bidirectional vehicle strings with measurement errors.
\newblock {\em IFAC Proceedings Volumes}, 47(3):9171--9176, 2014.

\bibitem{dai2017safety}
Siyuan Dai and Xenofon Koutsoukos.
\newblock Safety analysis of integrated adaptive cruise control and lane
  keeping control using discrete-time models of port-{H}amiltonian systems.
\newblock In {\em 2017 American Control Conference (ACC)}, pages 2980--2985.
  IEEE, 2017.

\bibitem{dai2020safety}
Siyuan Dai and Xenofon Koutsoukos.
\newblock Safety analysis of integrated adaptive cruise and lane keeping
  control using multi-modal port-{H}amiltonian systems.
\newblock {\em Nonlinear Analysis: Hybrid Systems}, 35:100816, 2020.

\bibitem{bansal2021port}
Harshit Bansal, P~Schulze, Mohammad~Hossein Abbasi, Hans Zwart, Laura
  Iapichino, Wil~HA Schilders, and Nathan van~de Wouw.
\newblock Port-{H}amiltonian formulation of two-phase flow models.
\newblock {\em Systems \& Control Letters}, 149:104881, 2021.

\bibitem{clemente2002geometric}
Jes{\'u}s Clemente-Gallardo, Ricardo Lopezlena, and Jacquelien~M.A. Scherpen.
\newblock Geometric discretization of fluid dynamics.
\newblock In {\em Proceedings of the 41st IEEE Conference on Decision and
  Control, 2002.}, volume~4, pages 4185--4190. IEEE, 2002.

\bibitem{rashad2021port}
Ramy Rashad, Federico Califano, Frederic~P Schuller, and Stefano Stramigioli.
\newblock Port-{H}amiltonian modeling of ideal fluid flow: Part {I}.
  {F}oundations and kinetic energy.
\newblock {\em Journal of Geometry and Physics}, 164:104201, 2021.

\bibitem{rashad2021portb}
Ramy Rashad, Federico Califano, Frederic~P. Schuller, and Stefano Stramigioli.
\newblock Port-{H}amiltonian modeling of ideal fluid flow: Part {II}.
  {C}ompressible and incompressible flow.
\newblock {\em Journal of Geometry and Physics}, 164:104199, 2021.

\bibitem{Sugiyama2008}
Yuki Sugiyama, Minoru Fukui, Macoto Kikuchi, Katsuya Hasebe, Akihiro Nakayama,
  Katsuhiro Nishinari, Shin ichi Tadaki, and Satoshi Yukawa.
\newblock Traffic jams without bottlenecks. experimental evidence for the
  physical mechanism of the formation of a jam.
\newblock {\em New Journal of Physics}, 10(3):033001, 2008.

\bibitem{tordeux2014collision}
Antoine Tordeux and Armin Seyfried.
\newblock Collision-free nonuniform dynamics within continuous optimal velocity
  models.
\newblock {\em Physical Review E}, 90(4):042812, 2014.

\bibitem{Helly1959}
Walter Helly.
\newblock Simulation of bottlenecks in single lane traffic flow.
\newblock In {\em Proceedings of the Symposium on Theory of Traffic Flow},
  pages 207--238, 1959.

\bibitem{jiang2001full}
Rui Jiang, Qingsong Wu, and Zuojin Zhu.
\newblock Full velocity difference model for a car-following theory.
\newblock {\em Physical Review E}, 64(1):017101, 2001.

\bibitem{ortega2001putting}
Romeo Ortega, Arjan Van Der~Schaft, Iven Mareels, and Bernhard Maschke.
\newblock Putting energy back in control.
\newblock {\em IEEE Control Systems Magazine}, 21(2):18--33, 2001.

\bibitem{Bando1995}
Masako Bando, Katsuya Hasebe, Akihiro Nakayama, Akihiro Shibata, and Yuki
  Sugiyama.
\newblock Dynamical model of traffic congestion and numerical simulation.
\newblock {\em Physical Review E}, 51(2):1035--1042, 1995.

\bibitem{teschl}
Gerald Teschl.
\newblock {\em Ordinary differential equations and dynamical systems}, volume
  140 of {\em Graduate Studies in Mathematics}.
\newblock American Mathematical Society, Providence, RI, 2012.

\bibitem{DapratoZab92}
Giuseppe Da~Prato and Jerzy Zabczyk.
\newblock {\em Stochastic equations in infinite dimensions}, volume~44 of {\em
  Encyclopedia of Mathematics and its Applications}.
\newblock Cambridge University Press, Cambridge, 1992.

\bibitem{AA02}
Yuri~A. Abramovich and Charalambos~D. Aliprantis.
\newblock {\em An invitation to operator theory}, volume~50 of {\em Graduate
  Studies in Mathematics}.
\newblock American Mathematical Society, Providence, RI, 2002.

\bibitem{pratoZabzyc96}
Giuseppe Da~Prato and Jerzy Zabczyk.
\newblock {\em Ergodicity for infinite-dimensional systems}, volume 229 of {\em
  London Mathematical Society Lecture Note Series}.
\newblock Cambridge University Press, Cambridge, 1996.

\bibitem{GS02}
Alison~L. Gibbs and Francis~Edward Su.
\newblock On choosing and bounding probability metrics.
\newblock {\em International Statistical Review}, 70(3):419--435, 2002.

\bibitem{Klenke04}
Achim Klenke.
\newblock {\em Probability theory}.
\newblock Universitext. Springer, London, second edition, 2014.
\newblock A comprehensive course.

\bibitem{KS91}
Ioannis Karatzas and Steven~E. Shreve.
\newblock {\em Brownian motion and stochastic calculus}, volume 113 of {\em
  Graduate Texts in Mathematics}.
\newblock Springer-Verlag, New York, second edition, 1991.

\bibitem{F46}
Evelyn Frank.
\newblock On the zeros of polynomials with complex coefficients.
\newblock {\em Bulletin of the American Mathematical Society}, 52:144--157,
  1946.

\bibitem{Tordeux2012}
Antoine Tordeux, Michel Roussignol, and Sylvain Lassarre.
\newblock Linear stability analysis of fisrt-order delayed car-following models
  on a ring.
\newblock {\em Physical Review E}, 86(3):036207, 2012.

\bibitem{tordeux2017influence}
Antoine Tordeux, Mohcine Chraibi, Andreas Schadschneider, and Armin Seyfried.
\newblock Influence of the number of predecessors in interaction within
  acceleration-based flow models.
\newblock {\em Journal of Physics A: Mathematical and Theoretical},
  50(34):345102, 2017.

\bibitem{cordes2023single}
Jakob Cordes, Mohcine Chraibi, Antoine Tordeux, and Andreas Schadschneider.
\newblock Single-file pedestrian dynamics: a review of agent-following models.
\newblock {\em Crowd Dynamics, Volume 4: Analytics and Human Factors in Crowd
  Modeling}, pages 143--178, 2023.

\bibitem{kloeden2011numerical}
Peter~E. Kloeden and Eckhard Platen.
\newblock {\em Numerical Solution of Stochastic Differential Equations}.
\newblock Stochastic Modelling and Applied Probability. Springer Berlin
  Heidelberg, 2011.

\end{thebibliography}
